\documentclass{aims}
\usepackage{amsmath}
  \usepackage{paralist}
  \usepackage{dsfont}
  \usepackage{array}
  \usepackage{graphics} 
  \usepackage{epsfig} 
 \usepackage[colorlinks=true]{hyperref}
\hypersetup{urlcolor=blue, citecolor=red}

\def\currentvolume{X}
 \def\currentissue{X}
  \def\currentyear{200X}
   \def\currentmonth{XX}
    \def\ppages{X--XX}

\renewcommand{\d}[0]{%
\ensuremath{\text{d}}}

\newcommand{\integ}[3]{%
\ensuremath{\displaystyle{\int^{#2}_{#1} #3}}}

\newcommand{\Div}{\nabla \cdot}

\newcommand{\abs}[1]{\lvert #1 \rvert}

\newcommand{\ra}{\rho_{1}}

\newcommand{\rb}{\rho_{2}}
\newcommand{\rs}{\rho}
\newcommand{\ri}{\rho_{i}}
\newcommand{\V}{\mathbf{U}}
\newcommand{\ve}{\mathbf{u}}

\newcommand{\Va}{\mathbf{U}_{1}}
\newcommand{\Vb}{\mathbf{U}_{2}}
\newcommand{\Vi}{\mathbf{U}_{i}}
\newcommand{\va}{\mathbf{u}_{1}}
\newcommand{\vb}{\mathbf{u}_{2}}
\newcommand{\ua}{\mathbf{v}_{1}}
\newcommand{\ub}{\mathbf{v}_{2}}
\newcommand{\da}{D_{1}}
\newcommand{\db}{D_{2}}
\newcommand{\di}{D_{i}}
\newcommand{\w}{\mathbf{w}}
\newcommand{\vv}{\mathbf{v}}

\newtheorem{theorem}{Theorem}[section]

\newtheorem{proposition}{Proposition}

\theoremstyle{definition}
\newtheorem{definition}[theorem]{Definition}
\newtheorem{remark}{Remark}

\def \w { \mathbf{w}}
\def \c {S}
\def \ONE {1}
\def \virg {\,,\,\,}
\def \nn {\mathbf{n}}
\def \vseq { \vspace{2mm}}
\def \T {\Omega}
 \def \b {\mu}
\def \C {C}
\title[A congestion Model for Cell migration]
      {A congestion Model for Cell migration}

\author[first-name1 last-name1 and first-name2 last-name2]{}


 \email{julien.dambrine@parisdescartes.fr}
 \email{Nicolas.Meunier@math-info.univ-paris5.fr}
 \email{Bertrand.Maury@math.u-psud.fr}
 \email{aude.roudneff@math.u-psud.fr}

\thanks{
The first and second authors are supported by the European Project ARTreat FP7 - 224297.
}

\begin{document}
\maketitle

\centerline{\scshape Julien Dambrine and Nicolas Meunier}
\medskip
{\footnotesize
 \centerline{MAP5, UFR de Math\'ematiques et Informatique }
\centerline{Universit\'e Paris Descartes, 45 rue des Saints-P\`eres, 75270 Paris Cedex 06 }
} 

\medskip

\centerline{\scshape Bertrand Maury and Aude Roudneff-Chupin}
\medskip
{\footnotesize
 \centerline{ Laboratoire de Math\'ematiques d'Orsay,
Universit\'e Paris-Sud,
91405 Orsay Cedex}
}

\bigskip

 \centerline{(Communicated by the associate editor name)}

\begin{abstract}
This paper deals with a class of macroscopic models for cell migration in a saturated medium for   two-species mixtures. Those species tend to achieve some motion according to a desired velocity, and congestion   forces them to adapt their velocity. This adaptation is modelled by a correction velocity which is chosen minimal in a least-square sense.
We are especially interested in two situations: a single active species  moves in a passive matrix (cell migration) with a given desired velocity, and a closed-loop Keller-Segel  type model, where the desired velocity  is the gradient of a self-emitted chemoattractant.

We propose a theoretical framework for the open-loop model (desired velocities are defined  as gradients of given functions) based on a formulation in the form of a gradient flow in the Wasserstein space.
We propose a numerical strategy to discretize the model, and illustrate its behaviour in the case of a prescribed velocity, and for the saturated  Keller-Segel model. 



\end{abstract}

\section{Introduction, modelling aspects}
\label{sec:model}

We propose a model  to handle 
 interactions between organisms such as unicellular organisms \textit{e.g.} bacteria or am\oe bia.
 %
  The behavior of an entity is based on a will to move with its desired velocity, regardless of others, but the fulfillment of individual wills is made impossible because of congestion. 
We consider here the case of two organisms (see Section~\ref{sec:conc} for a straightforward extension to several populations).
%
We describe the densities by measures $\ra$ and $\rb$.
We assume global saturation of the domain, i.e. $\ra+\rb = 1$.
%
In order to take into account the individual tendencies together with congestion constraints, we consider that individuals have two velocities, namely a \textit{desired} velocity that will be denoted by $\Va$ (\textit{resp.} $\Vb$) and a common  \textit{correction} velocity that will be denoted by $\w$. Densities $\ra$ and $\rb$  satisfy:
\begin{equation}\label{eq:model}
\partial _t \ri +\Div (\ri\left (\Vi+\w)\right)  =0\quad i=1,\, 2.\\
\end{equation}
We consider the correction velocity which minimizes $L^2$ norm among all those velocity fields which ensure preservation of the saturation constraint $\ra+\rb=1$, which can be expressed in a dual, Darcy-like, form:
\begin{equation}
\label{eq:darcy}
\left \{
\begin{array}{lcl}
\w + \nabla p &= &0 \\
\Div \w &=& -\Div \left ( \ra\Va + \rb\Vb\right ) .
\end{array}
\right .
\end{equation}




This basic model corresponds  to a competition between two species which tend to achieve a prescribed motion, and realize a sort of compromise. 
We shall be especially interested in the following  situations:
\begin{enumerate}

\item[(i)] Competition between species which tend to minimize a given function: $\Vi$ is defined as  $-\nabla\di$, with $\ra\Va+\rb\Vb$ not feasible (i.e. leading to violation of the saturation constraint) in general.
This setting, which extends the model proposed in~\cite{MauRouSan} to a two-population situation with distinct tendencies, shall be the core of the theoretical analysis proposed in section~\ref{sec:theo}.

\item[(ii)] Desired velocity for species $2$ is zero, and $\Va$ is prescribed as the gradient of a given function $\c$ (concentration of a chemoattractant).  Note that this case, which 
 corresponds to the migration of cells in a passive biological matrix, as encountered in the developpment of  atherosclerosis, is a particular case of  $(i)$.

\item[(iii)]  Keller-Segel model with congestion (closed-loop version of the basic model):
$\Vb$ is again $0$, and $\Va$ is defined as $\nabla \c$, where $\c$ is a chemotactic agent  created by the species $1$ itself. Assuming $\c$  diffuses in the mixture domain, we shall consider
$$
\partial _t \c - \Delta \c = \ra, 
$$
or, assuming diffusion is instantaneous, $-\Delta \c = \ra$.

\end{enumerate}

The question of boundary conditions rises delicate issues.  Setting equation in the whole space amounts to deal with infinite quantities of $1$ and/or $2$ (the sum of densities is equal to $1$), which rules out standard tools for theoretical  analysis and numerical simulation.
One may consider that both populations occupy a bounded, moving domain $\Omega(t)$,  on the boundary of which pressure is  zero. It is then natural to consider that the boundary of $\Omega$ moves with a normal velocity which identifies to the normal velocity of the mixture $(\ra \Va + \rb\Vb + \w)\cdot \nn$.
This approach is surely relevant in some situations, yet we shall 
 not consider it in this paper. Let us give an idea of one of the difficulties which are likely to occur: consider the monodimensional situation with initial condition and desired velocities:
$$
\ra =  \ONE_{(0,1)}\virg   \rb =  \ONE_{(-1,0)}\virg \Va = 1\virg \Vb = 0.
$$
Translation of those intervals at constant speed $1/2$ is obviously  a solution to our problem (with piecewise affine pressure, $0$ at both ends, $1/2$ at the interface).
Now consider the scenario where $1$ runs away with velocity $1$, and $2$  stays where it is (both species are separated, and the domain $\Omega(t)$ is no longer convex). One can even imagine that half of $1$ (initially located in $  \ONE_{(1/2,1)}$) takes off with velocity $1$, while the other half of $1$ trails $2$ at speed $1/3$.
 In this spirit, one can  create  an infinite number   of scenarios which all seem to be, in a reasonable sense, solutions to our problem. 
 Note that those problems are likely to occur in the case of expansion fields, as the one considered previously, because of the loss of monotonicity of the underlying evolution equation.

 Another approach consists in considering a fixed domain $\Omega$, delimited by rigid walls. This assumption is reasonable if one considers migration phenomena in living organisms.  The boundary condition for Darcy equations is of Neumann type: no-flux condition at walls  leads to $(\ra \Va + \rb\Vb + \w)\cdot \nn = 0$, i.e. 
 $$
\frac {  \partial p}{\partial n  } = (\ra \Va + \rb\Vb )\cdot \nn.
$$
This condition ensures no-flux for the mixture, but not for both species independently;  it may allow some out- or in-flux of either $1$ or $2$, which calls for further prescriptions on the boundary.
The situations we shall consider in next sections suggest that, at least in the case  $\Vb = 0$,  the model tends to create  pure zones (either $1$ only, or no $1$ at all) in the neighbourhood of boundaries. As a consequence, global no flux conditions turn  into  no flux conditions for both species. Yet, in 2 or 3 dimensions and for non zero desired velocities, one may have to consider more general situations. 
To overcome those difficulties, we shall consider in the theoretical part (Section~\ref{sec:theo}) the fully periodic setting.

In a different context, such saturated two-component mixtures have been studied in~\cite{otto97}, with a motion driven by chemical potentials.
In the biological context, 
a Keller-Segel model with logistic sensitivity (KSLS) has been proposed recently~\cite{dolak,DalibardPerthame}. This KSLS model reads
\begin{equation}
\label{eq:KSLS}
\partial _t \rs +\Div \left (\rs\left (1-\rs\right)\V\right) =0, 
\end{equation}
where the ``desired'' velocity $\V$ is of the chemotactic type:
 $$
\V = \nabla S \virg \partial_t S - \Delta S = \rho.
 $$

First of all, in one dimension (say $\Omega = (0,1)$) with  no-flux boundary conditions at $x=1$, our model reads
$$
\partial _t \rs+\partial_x \left( \rs(U-\partial _x p)\right)  =0\,, 
$$
$$
-\partial _{xx}p= -\partial_x (\rs U)
$$
with  (no flux condition) $\rs  U- \partial _x p = 0$ at $1$, so that $\partial_x p = \rho U$ in the whole interval, 
 and therefore it   identifies exactly with  KSLS. 
 Note that in case the velocity is given and constant, we recover an inviscid Burgers equation.
 
Both models are different for $d \geq 2$. Yet, they present some common structure which can be described as follows:
consider (situation $(iii)$ above)  the case of zero velocity for species 2, and velocity for 1 given as the gradient of a density $\c$ of chemoattractant created by $1$  itself.
We consider  the biperiodic setting, and define $\Delta ^{-1}$ as the operator which maps a function $g$ with zero mean value over the biperiodic domain onto the zero mean value solution  to 
$$
\Delta u = g.
$$
The pressure in our model  can then be written $p = \Delta^{-1} \Div  (\rs\V)$, so that transport equation for species $1$ becomes
$$
\partial _t \rs +\Div \left (\rs\left (\V -   \nabla   \Delta^{-1} \Div  (\rs\V)           \right)\right) =0.
$$
which is formally equivalent to KSLS model~\eqref{eq:KSLS} with identity operator replaced by  Helmoltz projection $ \nabla   \Delta^{-1} \Div $ (projection onto the space of irrotational fields).
Nonlocality of the latter operator differentiate both models. Yet both PDE's systems are quite similar from the spectral point of view. 

As for modelling aspects, they rely on different assumptions:
KSLS model can be seen as standard KS model with a reduced velocity $\V = (1-\rs) \nabla \c$, which can express for example the fact that when entities reach some critical density (set to $1$ here), they disturb each other and are no longer able to sense the underlying gradient. 
In our approach, we consider that entities continue to exert some action which would lead them in the right direction if they were alone, but they are (mechanically) prevented from fulfilling their purpose, because of the presence of other entities (possibly with different ``strategies''). The model is in some way less constrained as motion of saturated  zones is possible, but also more constrained because the other species (which replaces empty space in KSLS) has to be swept away, at some price,  by  species $1$.

\section{Theoretical  framework}
\label{sec:theo}


We consider in this section the case of $2$ species in the flat torus $\T = S^1\times S^1$, with desired velocities prescribed as follows
$$\Va = - \nabla \da, \; \Vb = - \nabla \db, $$
where $\da$ and $\db$ are given smooth functions over $\T$. 
The system corresponding to this situation writes
\begin{equation}
\left\{ \begin{array}{rcl}
\partial_{t} \ra + \nabla \cdot (\ra (\Va + \w)) & = & 0\vseq\\
\partial_{t} \rb + \nabla \cdot (\rb (\Vb + \w))  & = & 0\vseq\\
\w &=& -\nabla p \vseq\\
-\Delta p & = & - \nabla \cdot (\ra \Va + \rb \Vb).
\end{array} \right.
\label{eq:demo}
\end{equation}

Although concentrations for both species automatically admit a density which is in $L^{\infty}$ (thanks to the congestion constraint), we shall keep the general setting of measures, which is usually adopted in optimal transport, and which allows for generalisations (e.g. relaxing the congestion constraint in some zones). Besides, we disregard normalization to obtain probability measures:
we shall consider that $\rho_1\otimes \rho_2$ belongs to 
$\mathcal{P}(\T\times \T)$, although the total mass is not $1$. 
Similarly, we consider that $\rho_i \in \mathcal{P}(\T)$.

In what follows, regularity of functions defined over $\T$ accounts for periodic conditions.
In particular  $H^1(\T)$ is defined as the set of biperiodic functions with $H^1$ regularity (closure for the  $H^1$ norm of biperiodic regular functions).

\begin{definition}[Weak solutions]
We say that $(\ra, \rb)$ is a weak solution of~\eqref{eq:demo} with initial condition $(\ra^0, \rb^0)$ if  $\ra+\rb = 1$ for a.e. $t$, and if 
there exists $p \in H^1(\Omega)$ such that for all $\varphi_{1}, \varphi_{2} \in \mathcal{C}_{c}^{\infty}([0,T[ \times \T)$,  for a.e. $t\in (0,T)$,  and for all $q \in H^1(\T)$, we have
$$\int_{0}^T \int_{\T} \left( \partial_{t} \varphi_{1} + \nabla \varphi_{1} \cdot (\Va - \nabla p) \right) d\ra \; + \; \int_{\T} \varphi_{1}(0,.) d\ra^0 \; = \; 0$$
$$
\int_{0}^T \int_{\T} \left( \partial_{t} \varphi_{2} + \nabla \varphi_{2} \cdot (\Vb - \nabla p) \right) d\rb \; + \; \int_{\T} \varphi_{2}(0,.) d\rb^0 \; = \; 0,
$$
$$\int_{\T} \nabla p \cdot \nabla  q  \; = \; \int_{\T} \nabla q \cdot \Va \, d\ra \; + \; \int_{\T} \nabla q \cdot \Vb \, d\rb.
$$
\label{weak:sol}
\end{definition}

Let $\rho = \ra \otimes \rb
$ for a.e. $t$, $\V = (\Va, \Vb)$. We can rewrite the previous formulation as follows: for all $\varphi \in \mathcal{C}_{c}^{\infty}([0,T[ \times (\T \times \T))$ 
$$
\int_{0}^T \int_{\T \times \T} (\partial_{t} \varphi + 
 \langle \V-(\nabla p, \nabla p), \nabla \varphi \rangle
) \, d\rho \; + \; \int_{\T \times \T} \varphi(0,.) \, d\rho^0 \; = \; 0.
$$
  and for a.e. $t\in (0,T)$, all $q \in H^1(\T)$, 
$$\int_{\T} \nabla p \cdot \nabla q  \; = \; \int_{\T \times \T} \langle (\nabla q, \nabla q), \V \rangle \, d\rho.
$$

Let us state the main result of this section:
\begin{theorem}[Existence of a solution]
If $\da$ and $\db$ are Lipschitz functions, and $(\ra^0, \rb^0)$ positive measures that satisfy $\ra^0 + \rb^0 = 1$ a.e., then equation \eqref{eq:demo} admits at least a weak solution, according to definition \ref{weak:sol}.
\label{thm:existence}
\end{theorem}

The proof of this theorem 
relies on the notion of gradient flow in the Wasserstein space (see e.g.~\cite{Amb}). 
Theoretical analysis of a similar problem in the context of  crowd motions was proposed in~\cite{MauRouSan}. The proof proposed therein is quite general, and the same approach could be carried out in the present context. Yet, to avoid some technicalities and to give a clearer view of the underlying gradient flow structure, we propose here an alternative approach which is directly based on the main existence and characterization theorems in~\cite{Amb}. The rest of this section describes the main outlines of this approach, which is based on theories of optimal transport and gradient flows, which we will use to prove theorem~\ref{thm:existence}. For more details, see the books of Villani \cite{Vil1}, \cite{Vil2}, and Ambrosio et al.~\cite{Amb}, \cite{AmbSav}. See  also the recent application of this framework~\cite{OttoGigli} 
to recover entropic solutions to the Burgers' equation (which identifies in some way to the proposed model for $d=1$, as pointed out in the introduction).

We 
define a discrete scheme for~\eqref{eq:demo}:
\begin{definition}[JKO scheme]{}
Let $\tau > 0$ be given, and $\rho^0 = \ra^0 \otimes \rb^0 \in \mathcal{P}(\T\times \T)$ an initial density. We define $\rho^1, \ldots, \rho^n$ (we drop the dependence upon $\tau$ to alleviate notations) recursively according to
\begin{equation} 
\rho^n \in \mathop{\textmd{ argmin }}\limits_{\rho \in \mathcal{P}(\T\times \T)} \left (J(\rho) + I_{K}(\rho) + \dfrac{1}{2 \tau} W_{2}^2(\rho, \rho^{n-1}) \right ) 
\label{eq:JKO}
\end{equation}
where $J$ is given by
$$J(\rho) = \int_{\T \times \T} (\da(x_1) + \db(x_2)) d \rho (x_1, x_2)$$
and $K$ is the set of admissible densities
$$K = \{ \mu \in \mathcal{P}(\T \times \T) : \mu = \mu_{1} \otimes \mu_{2}, \; \mu_{1} + \mu_{2} = 1 \textmd{ a.e.} \}.
$$
The notation $I_{K}$ stands for the indicatrix function of $K$, i.e.
$$I_{K}(\rho) = \left\{ \begin{array}{cl}
0 & \textmd{ if } \rho \in K\\
+\infty & \textmd{ if } \rho \not \in K.
\end{array} \right.$$
\end{definition}
This is a well-known scheme in gradient flow theory (see \cite{DeG}, \cite{JorKinOtt}, \cite{BenBre}, \cite{Amb}, \cite{AmbSav}), which has been widely used to prove existence theorems. Taking a minimizing sequence of the right-hand side of \eqref{eq:JKO}, we can verify that every of these minimizing problems admit at least a solution.

Let us underline that here the JKO scheme is only a tool to prove the existence of a solution to our problem, and has not been used for numerical purposes. As a matter of fact, the numerical resolution at each time step of the minimisation problem leads to many difficulties, and we have chosen a totally different approach for the numerical tests in section \ref{sec:num}.

\begin{proposition}[Distance between two product measures]{}
If $\mu = \mu_{1} \otimes \mu_{2}, \nu = \nu_{1} \otimes \nu_{2}$, then 
$$
W_{2}^2(\mu, \nu) = W_{2}^2(\mu_{1}, \nu_{1}) + W_{2}^2(\mu_{2}, \nu_{2}).
$$
Moreover, if $\mathbf{r}_{1}$ (resp. $\mathbf{r}_{2}$) is the optimal transport between $\mu_{1}$ and $\nu_{1}$ (resp. $\mu_{2}$ and $\nu_{2}$), then $\mathbf{r} = (\mathbf{r}_{1}, \mathbf{r}_{2})$ is the optimal transport between $\mu$ and $\nu$.
\end{proposition}
\begin{proof} 
As $(\mathbf{r}_{1}, \mathbf{r}_{2})$ transports $\mu$ to $\nu$, the previous definition of $W_{2}$ gives $W_{2}^2(\mu, \nu) \leq W_{2}^2(\mu_{1}, \nu_{1}) + W_{2}^2(\mu_{2}, \nu_{2})$. To obtain the converse inequality, we simply use the dual definition of $W_{2}$ with Kantorovich potentials (see \cite{Vil1}).
\end{proof}

Thanks to the previous proposition, the JKO scheme can be rewritten as follows
\begin{equation}
(\ra^n, \rb^n) \in \mathop{\textmd{ argmin }}\limits_{\ra + \rb = 1} \left( \int_{\T} \da d\ra + \int_{\T} \db d\rb + \dfrac{1}{2 \tau} W_{2}^2(\ra, \ra^{n-1}) + \dfrac{1}{2 \tau} W_{2}^2(\rb, \rb^{n-1}) \right)
\label{eq:JKO:reform}
\end{equation}


\begin{proposition}[Convergence of the JKO scheme]
Let $\rho_{\tau} $ be the piecewise constant interpolation (in time) of the sequence $(\rho^n)_{n}$ defined by the JKO scheme\eqref{eq:JKO} for the time step $\tau$. Under the assumptions of theorem \ref{thm:existence}, there exists a subsequence of $(\rho_{\tau})_{\tau}$ which converges weakly to a weak solution of the transport equation
$$
\partial_{t} \rho + \nabla \cdot (\rho \ve) = 0$$
where $\ve$ verifies
$$
\ve \in - \partial (J + I_{K})(\rho).
$$
\label{prop:convergence}
\end{proposition}

\begin{proof}
We have to verify that the functionnal $\Phi := J + I_{K}$ satisfies the assumptions of the theory developed by Ambrosio et al. in \cite{Amb}. $\Phi$ is clearly proper and lower semicontinuous, and its sublevels $[\Phi(\rho) \leq c]$ are compact. Moreover, we can prove that $\Phi$ is regular according to definition 10.1.4 in \cite{Amb}, i.e. for all sequence $\rho^n = (\ra^{n} \otimes \rb^n) \in K$, for all  $\mathbf{u}^{n} \in \partial \Phi(\rho^n)$ such that $(\rho^n)$ narrowly converges towards $\rho \textmd{ in } \mathcal{P}(\T \times \T)$, and
$$\mathop{\sup}\limits_{n} ||\mathbf{u}^{n}||_{L^2(\mu_{n})} < + \infty, \quad \mathbf{u}^{n} \rightharpoonup \mathbf{u} \textmd{ weakly }$$
then $\mathbf{u} \in \partial \Phi (\rho)$.
All assumptions of Prop. 2.2.3, Th. 2.3.1, and Th. 11.1.3 of \cite{Amb} are satisfied, therefore there exists a subsequence of $(\rho_{\tau})_{\tau}$ which converges to a weak solution of 
$$\partial_t \rho + \nabla \cdot (\rho \ve) = 0,$$
with
$$\ve \in - \partial \Phi (\rho) \quad \hbox{for a.e. } t.
$$
\end{proof}

\begin{proposition}
The limit of $(\rho_{\tau})_{\tau}$ in proposition \ref{prop:convergence} is actually a weak solution of equation \eqref{eq:demo}.
\end{proposition}

\begin{proof}
It is quite easy to verify that when $\tau$ converges to $0$, the properties $\rho(t,.) = \ra(t,.) \otimes \rb(t,.)$ with $\ra(t,x) + \rb(t,x) = 1$ for a.e. $x$, and $\ve(x,y) = (\va(x), \vb(y))$ still hold true. Therefore, the continuity equation rewrites
\begin{equation}
\left\{ \begin{array}{rcl}
\partial_t \ra + \nabla \cdot (\ra \va) & = & 0\\
\partial _t\rb + \nabla \cdot (\rb \vb) & = & 0
\end{array} \right. ,
\label{eq:continuity}
\end{equation}
and the condition $\ra + \rb = 1$ gives the following equation on $\va$ and $\vb$
$$\nabla \cdot (\ra \va + \rb \vb) = 0.$$

We now use the fact that $- \ve$ is a strong subdifferential, i.e. that for all transport map $\mathbf{r}$, we have
$$\Phi(\rho) - \int_{\T \times \T} \langle \ve, \mathbf{r} - \mathbf{id} \rangle \rho \; \leq \; \Phi(\mathbf{r}_{\#} \rho) + o\left(||\mathbf{r} - \mathbf{id} ||_{L^2(\rho)}\right).$$
Let us underline that if $\mathbf{r}_{\#} \rho \not \in K$, the previous inequality does not give any information, as $\Phi(\mathbf{r}_{\#} \rho) = + \infty$.

We define the set of admissible velocities as follows
$$
C_{\rho} := \{ \vv = (\ua, \ub) \in L^2(\ra) \times L^2( \rb) : \nabla \cdot (\ra \ua + \rb \ub) = 0 \}.
$$
We just proved that $\ve \in C_{\rho}$. Let $\mathbf{v} \in C_{\rho}$, $\varepsilon \not = 0$, and $\mathbf{r}^{\varepsilon} = \mathbf{id} + \varepsilon \mathbf{v}$. In most cases, we have $\mathbf{r}^{\varepsilon}_{\#}\rho \not\in K$. However, it is possible to find a transport $\mathbf{t}^{\varepsilon}$ such that $(\mathbf{t}^{\varepsilon} \circ \mathbf{r}^{\varepsilon})_{\#} \rho \in K$, and $W_{2}((\mathbf{t}^{\varepsilon} \circ \mathbf{r}^{\varepsilon})_{\#} \rho, \mathbf{r}^{\varepsilon}_{\#} \rho) = o(\varepsilon)$. 
The subdifferential inequality applied to $\mathbf{t}^{\varepsilon} \circ \mathbf{r}^{\varepsilon}$ gives
$$
\hspace*{-3cm} \int_{\T \times \T} (\da + \db)\,d \rho -  \int_{\T \times \T} \langle \ve, \mathbf{t}^{\varepsilon} \circ \mathbf{r}^{\varepsilon} - \mathbf{id} \rangle \,d\rho 
$$
$$\hspace*{4cm} \leq \; \int_{\T \times \T} (\da + \db) (\mathbf{t}^{\varepsilon} \circ \mathbf{r}^{\varepsilon})_{\#}  \rho + o\left(||\mathbf{t}^{\varepsilon} \circ \mathbf{r}^{\varepsilon} - \mathbf{id} ||_{L^2(\rho)}\right).
$$
Using estimates between $(\mathbf{t}^{\varepsilon} \circ \mathbf{r}^{\varepsilon})_{\#} \rho$ and $\mathbf{r}^{\varepsilon}_{\#} \rho$, we get
$$\displaystyle \int_{\T \times \T} (\da + \db) \, d\rho -  \int_{\T \times \T} \langle \ve, \mathbf{r}^{\varepsilon} - \mathbf{id} \rangle \,d\rho \; \leq \;   \displaystyle \int_{\T \times \T} (\da + \db) \mathbf{r}^{\varepsilon}_{\#}  \rho + o(\varepsilon),$$
which implies
$$ \displaystyle  \int_{\T \times \T} \langle - \nabla D - \ve, \varepsilon \mathbf{v} \rangle\, d \rho \; \leq \;  \displaystyle o(\varepsilon).$$
Letting $\varepsilon$ go to $0$ from positive and negative values, we obtain
$$\int_{\T \times \T} \langle - \nabla D - \ve, \mathbf{v} \rangle \, d\rho \; = \;  0,$$
which means that $- \nabla D - \ve \in C_{\rho}^{\perp}$. As $\ve \in C_{\rho}$, $\ve$ is indeed the projection of $- \nabla D = (\Va, \Vb) $ onto  $C_{\rho}$, i.e.
$$
\ve = \mathop{\textmd{ argmin }}\limits_{\mathbf{v} \in C_{\rho}} \int_{\T \times \T} |\V - \mathbf{v}|^2 d\rho.
$$
This minimizing problem can be written as a saddle-point problem for which we can prove existence and uniqueness of a solution $(\ve, p)$ which satisfies
$$\ve + (\nabla p, \nabla p) = \V,$$
i.e.
$$\left\{ \begin{array}{rcl}
\va = \Va - \nabla p\vseq\\
\vb = \Vb - \nabla p
\end{array} \right. .$$
Moreover, since $\nabla \cdot (\ra \va + \rb \vb) = 0$, the equation on $p$ is given by
$$-\Delta p =- \nabla \cdot (\ra \nabla p + \rb \nabla p) =- \nabla \cdot (\ra \Va + \rb \Vb).$$
If we replace the expression of $\ve$ in equations \eqref{eq:continuity}, we finally get 
$$\left\{ \begin{array}{l}
\partial_{t} \ra + \nabla \cdot (\ra (\Va - \nabla p)) \; = \; 0\vseq\\
\partial_{t} \rb + \nabla \cdot (\rb (\Vb - \nabla p))  \; = \; 0\vseq\\
-\Delta p \; = \;  - \nabla \cdot (\ra \Va + \rb \Vb)
\end{array} \right. ,$$
i.e. $\rho$ is a solution of~\eqref{eq:demo}.
\end{proof}
 

\begin{remark} (Uniqueness)
We focused here on the proof of existence of a solution. Under reasonable assumptions on $D_1$ and $D_2$, it is to be expected that the JKO process leads to a unique solution (see~\cite{MauRouSan} for remarks regarding uniqueness in a similar setting). Yet, as we pointed out in the introduction, the system is under some conditions equivalent  to the standard inviscid Burgers' equation, which rules out uniqueness in general.
 It is  natural to wonder whether
 the JKO scheme selects a particular solution, namely the entropic one. This delicate question is still widely open, but a recent work on Burgers equation (\cite{OttoGigli}) suggests  a positive answer.
 \end{remark}


\section{Numerical methods and results}
\label{sec:num}
We consider here the transport  model  with congestion, with one active species only ($\Va = \V$, $\Vb = 0$) expressed in terms of the  active species only:
\begin{align}
 \partial_t \rs + \nabla \cdot ( \rs\, (\V  + \w)) & = 0 \, , \label{Num-1} \vseq\\
 \w &=-\nabla p \, , \label{Num-2} \vseq\\
- \Delta p &  = - \nabla \cdot ( \rs\V \, ) \, . \label{Num-3}
\end{align}
\subsection{Discretization and maximum principle}
First, a time discretization of the system \eqref{Num-1}-\eqref{Num-3} is given. We set
$
t^n = n \,  \delta t $, and
\begin{equation}
\rs^n(x) \sim \rs(t^n,x) \, , \quad  \w^n(x) \sim \w(t^n,x) \, ,  \quad  \V^n(x) \sim \V(t^n,x) \, , \quad p^n(x) \sim p(t^n,x) \,  . \notag
\end{equation}
Equation~\eqref{Num-1} is approached by using a backwards Euler finite difference method. Semi-discretized version of System~\eqref{Num-1}-\eqref{Num-3} hence reads:
\begin{eqnarray}
 \rs^{n+1} &= &\rs^n  - \delta t \,  \nabla \cdot ( \rs^n \, (\V^n + \w^n)) \, , \label{Num-4} \vseq\\
 \w^n&=& - \nabla p^n \, , \label{Num-5} \vseq\\
- \Delta p^n &= & -\nabla \cdot ( \rs^n \, \V^n) \, . \label{Num-6}
\end{eqnarray}
For the sake of simplicity, we only give the 1D spatial discretization of~\eqref{Num-4}-\eqref{Num-6}, the extension to 2D or 3D on cartesian grids being straightforward.

 We introduce:
\begin{equation}
x_i = i \,  \delta x \, , \quad  i=1, \dots, N_x \, . \notag
\end{equation}
Since the equations of the model are written in a conservative form, the natural framework to be used for the spatial discretization of \eqref{Num-4}-\eqref{Num-6} is the finite volume framework. 
We hence introduce the control volume defined by:
\begin{equation}
\C_{i}=\left (  x_{i-\frac{1}{2}} \, , \,  x_{i+\frac{1}{2}} \right ) \, . \notag
\end{equation}
We denote by $(\rho^n_i)_{i}$ and $(p_i^n)_i$ the piecewise constant approximations  of densities and pressures at time $t^n$ (i.e. $\rho^n_i$ stands for the value of $\rho$ at cell center $x_i$, at time $t^n$). As for flux variables, we define approximate values at interfaces: $\w_i^n$ (resp.  $\V_i^n$) stands for  $\w$ (resp. $\V$) at interface $x_{i-\frac{1}{2}}$,  

By integrating~\eqref{Num-4} over $\C_{i}$ and by using the divergence theorem, we obtain
\begin{equation}
\integ{\C_{i}}{}{\rs^{n+1}\, \d x} = \integ{\C_{i}}{}{\rs^{n}\, \d x}  - \delta t \, \left[ \rs^n  (\V^n + \w^n) \right]^{x_{i+\frac{1}{2}}}_{x_{i-\frac{1}{2}}} \, .  \label{Num-5pa}
\end{equation}
We approximate the integral terms in the following way:
\begin{equation}
\integ{\C_{i}}{}{\rs^{n+1}\, \d x} \sim \rs_{i}^{n+1} \, \delta x \, , \quad \integ{\C_{i}}{}{\rs^{n}\, \d x} \sim \rs_{i}^{n} \, \delta x \, . \notag
\end{equation}
An upwind discrete flux is then introduced  to approximate the remaining flux term in \eqref{Num-5pa}: 
 \begin{align}
  \rs^n(x_{i-\frac{1}{2}})  \left( \V^n(x_{i-\frac{1}{2}}) + \w^n(x_{i-\frac{1}{2}}) \right) \sim  A^{up} \left( \V^n_i ,\rs^n_{i-1},\rs^n_i \right) + A^{up} \left( \w^n_i,\rs^n_{i-1},\rs^n_i \right) \label{Num-6pp} 
 \end{align}
 where the numerical flux $A^{up}$ is defined by
\begin{equation}
A^{up}(u,\, \rs^-, \, \rs^+) =
\begin{cases}
u\, \rs^-  \;\;\;\; \text{if } u>0   \, ,  \notag \\
u\, \rs^+  \;\;\;\; \text{if } u<0  \, . \notag
\end{cases}
\end{equation}
We finally obtain ;
\begin{align}
\rs^{n+1}_i =\rs^{n}_i &-\dfrac{\delta t}{\delta x}  \left( A^{up} \left( \V^n_{i+1} ,\rs^n_i,\rs^n_{i+1} \right) - A^{up} \left( \V^n_i ,\rs^n_{i-1},\rs^n_i \right) \right) \notag \\ 
&-\dfrac{\delta t}{\delta x}  \left( A^{up} \left( \w^n_{i+1} ,\rs^n_i,\rs^n_{i+1} \right) - A^{up} \left( \w^n_i ,\rs^n_{i-1},\rs^n_i \right) \right)\, . \label{Num-7} 
\end{align}
It is important to notice that  fluxes $\rs \V$ and $\rs \w$ were treated separately in~\eqref{Num-6pp}. As we will see, this plays an essential role in preserving the maximum principle on $\rs$ for the numerical solution. This advection scheme is stable under the Courant-Friedrichs-Lewy condition :
\begin{equation}
\delta t <\dfrac{1}{2} \cdot \dfrac{\delta x}{ \abs{\V^n}_{\infty} + \abs{\w^n}_{\infty} } \label{Num-7p} \, .
\end{equation} 
Eq.~\eqref{Num-6} is discretized in space with
\begin{equation}
\dfrac{p^n_{i+1}-2 \,p^n_{i} + p^n_{i-1} }{\delta x^2} =\dfrac{1}{\delta x}  \left (A^{up} \left( \V^n_{i+1} ,\rs^n_i,\rs^n_{i+1} \right) - A^{up} \left( \V^n_i ,\rs^n_{i-1},\rs^n_i \right) \right )  \, .  \label{Num-8} 
\end{equation}
Finally, by using an Euler finite difference scheme in~\eqref{Num-5},  correction velocity $\w$ is approximated with
\begin{equation}
\w_i ^n=-  \dfrac{p^n_{i}-p^n_{i-1}}{\delta x} .  \label{Num-9}  \\
\end{equation}
\begin{proposition}
The numerical scheme~\eqref{Num-7}-\eqref{Num-9} satisfies the following  maximum principle: 
\begin{align}
 \text{if } \; & 0 \leq \rs^0_{i} \leq 1\, ,  \, \, \forall \, i=1,\dots,N_x \, \notag \\
 \text{then } \; & 0 \leq \rs^n_{i} \leq 1\, , \, \, \forall \,  i=1,\dots,N_x \quad  \forall \, n  .\notag
\end{align}
\end{proposition}
\begin{proof}
For the sake of simplicity, only periodic boundary conditions are considered in the following proof:
\begin{align}
& \rs^n_1 = \rs^n_{N_x} \, , \notag \\
& p^n_1 =p^n_{N_x}\, . \notag
\end{align}
Note that, in practical, the proposition remains true if the no-flux boundary condition (mentioned in Section 1) is taken over $p$. In this case, appropriate boundary conditions have to be taken upon $\rs$ in the upwind scheme:
\begin{align}
& A^{up} \left( u_1 ,\rs_0,\rs_1 \right) = 0  \, , \notag \\
& A^{up} \left( u_{N_x+1} ,\rs_{N_x},\rs_{N_x+1} \right) = 0 \, . \notag
\end{align}
First the positivity of the numerical scheme is given by the well known positivity of the upwind scheme under the C.F.L. condition \eqref{Num-7p}. For the upper bound $\rs^n_{i} \leq 1$, the proof relies on the positivity of $\b^n_i := 1-\rs^n_i$, that is given by a similar numerical scheme. Indeed, equation \eqref{Num-7} implies
\begin{align}
\b^{n+1}_i =\b^{n}_i &+\dfrac{\delta t}{\delta x}  \left( A^{up} \left( \V^n_{i+1} ,\rs^n_i,\rs^n_{i+1} \right) - A^{up} \left( \V^n_i ,\rs^n_{i-1},\rs^n_i \right) \right) \notag \\ 
&+\dfrac{\delta t}{\delta x}  \left( A^{up} \left( \w^n_{i+1} ,1-\b^n_i,1-\b^n_{i+1} \right) - A^{up} \left( \w^n_i ,1-\b^n_{i-1},1-\b^n_i \right) \right)\, . \notag 
\end{align}
Since $A^{up}(u, 1-\rs^+,1-\rs^-)=u-A^{up}(u,\rs^+,\rs^-)$ we have
\begin{align}
\b^{n+1}_i =\b^{n}_i &+\dfrac{\delta t}{\delta x}  \left( A^{up} \left( \V^n_{i+1} ,\rs^n_i,\rs^n_{i+1} \right) - A^{up} \left( \V^n_i ,\rs^n_{i-1},\rs^n_i \right) \right) + \dfrac{\delta t}{\delta x} \left( \w^n_{i+1} - \w^n_{i} \right)  \notag\\
&- \dfrac{\delta t}{\delta x}  \left( A^{up} \left( \w^n_{i+1} ,\b^n_i,\b^n_{i+1} \right) - A^{up} \left( \w^n_i ,\b^n_{i-1},\b^n_i \right) \right)\, . \notag
\end{align}
By replacing the values of $\w^n_i$ with the discrete gradient of $p$ as shown in \eqref{Num-9}, one gets
\begin{align}
\b^{n+1}_i =\b^{n}_i &+\dfrac{\delta t}{\delta x}  \left( A^{up} \left( \V^n_{i+1} ,\rs^n_i,\rs^n_{i+1} \right) - A^{up} \left( \V^n_i ,\rs^n_{i-1},\rs^n_i \right) \right) \notag \\
&- \dfrac{\delta t}{\delta x} \left( \dfrac{p^n_{i+1} - 2\, p^n_i + p^n_{i-1}}{\delta x}  \right) \notag \\
&- \dfrac{\delta t}{\delta x}  \left (  A^{up} \left( \w^n_{i+1} ,\b^n_i,\b^n_{i+1} \right) - A^{up} \left( \w^n_i ,\b^n_{i-1},\b^n_i \right) \right ) \, . \label{Num-11} 
\end{align}
We recognize in the first terms of the right hand side of \eqref{Num-11} the exact discrete expression of \eqref{Num-8}. Hence the numerical scheme satisfied by $\b$ is
\begin{equation}
\b^{n+1}_i =\b^{n}_i - \dfrac{\delta t}{\delta x}  \left ( A^{up} \left( \w^n_{i+1} ,\b^n_i,\b^n_{i+1} \right) - A^{up} \left( \w^n_i ,\b^n_{i-1},\b^n_i \right) \right ) \, , \notag 
\end{equation}
which is an upwind discretization of the advection problem: $\partial_t \b + \nabla \cdot (\b\,  \w) = 0$. Thanks to the positivity of the upwind scheme under the C.F.L. condition \eqref{Num-7p} we deduce the discrete maximum principle for $\rs$.
\end{proof}
\subsection{Numerical results}
In this section we present several numerical simulations performed for  the migration model~\eqref{Num-1}-\eqref{Num-3} with the numerical scheme~\eqref{Num-7}-\eqref{Num-9} introduced above. First, in order to validate the numerical method, two 1D test cases for which the exact solution is known are presented. Then 2D simulations are shown for two types of situations: first the case where the desired velocity $\V$ is known, and then the case where $\V$ is given as a function of $\rs$.
\subsubsection{1D simulations}
In this section the domain is the interval $(0, 1)$, the desired velocity is $\V=1$, and we consider the following boundary conditions for $p$:
\begin{equation}
p^n(x=0)=0 \, , \quad \partial_x p^n(x=1)= \rs^n(1) \, \V^n(1) \, . \label{1Dim-1a}
\end{equation}
Two different initial conditions are considered:
\begin{equation}
\rs^0=\dfrac{1}{2} \mathbf{1}_{[0.1 \, ,\,  0.9]} + \mathbf{1}_{[0.9 \, ,\,  1]} \quad \label{1Dim-1}
\hbox{ (see Fig.~\ref{fig:1Dim-1}a) }, 
\end{equation}
and
\begin{equation}
\rs^0=\mathbf{1}_{[0.3 \, ,\,  0.5]} \quad \label{1Dim-2}
\hbox{(see Fig.~\ref{fig:1Dim-2}a).}
\end{equation}
\begin{figure}
\begin{center}
\begin{tabular}{ll}
\textbf{a)}  & \textbf{b)}  \\
  \includegraphics[width=6.4cm]{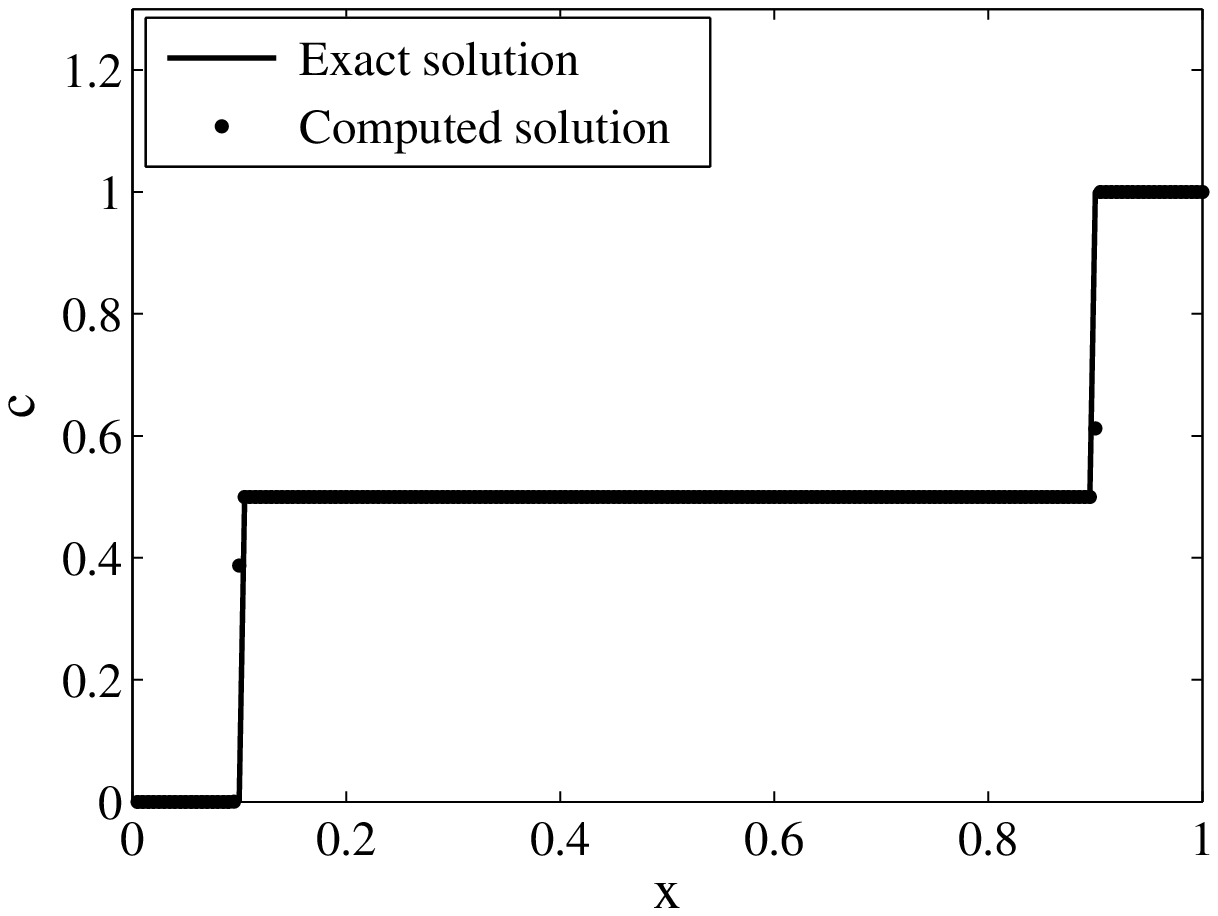}  & \includegraphics[width=6.4cm]{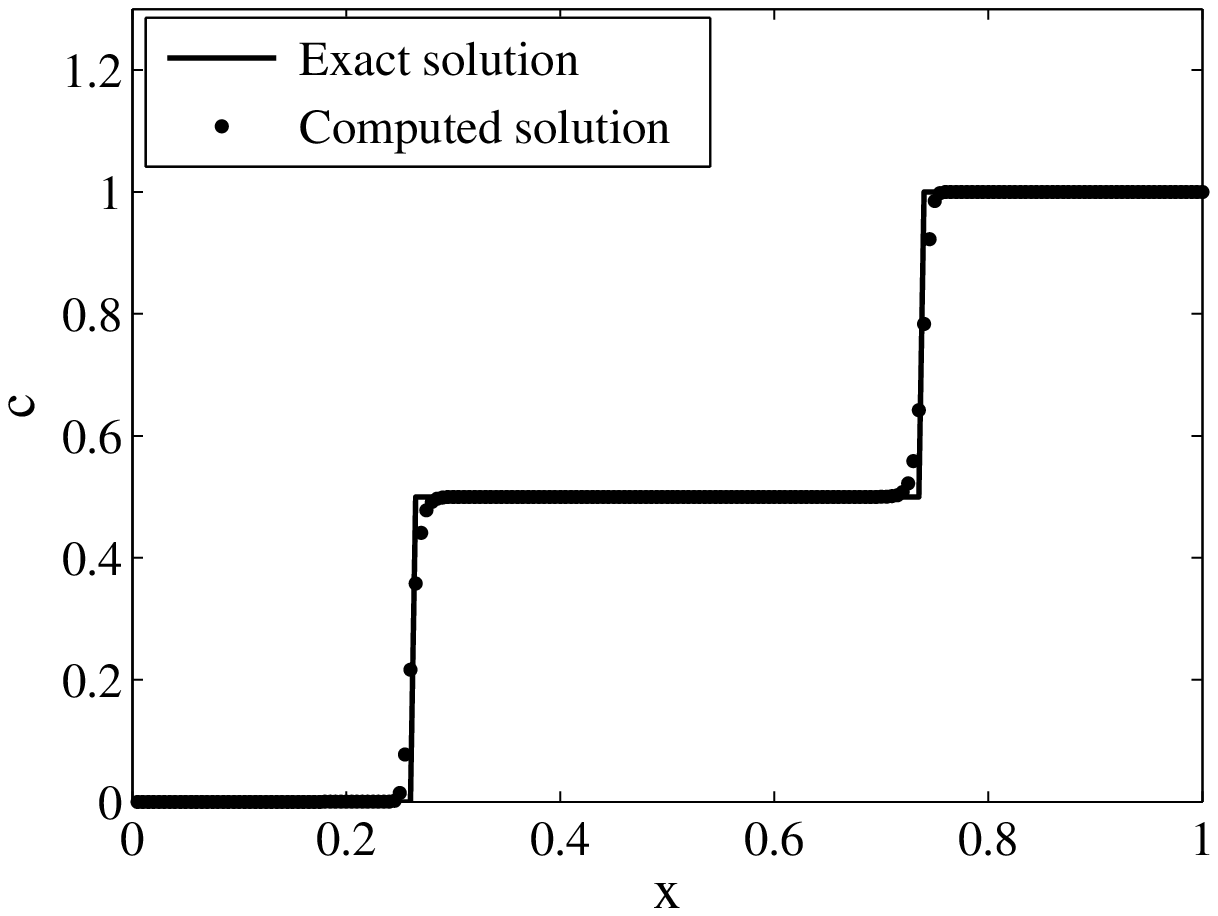} \\
  \textbf{c)}  & \textbf{d)}  \\
  \includegraphics[width=6.4cm]{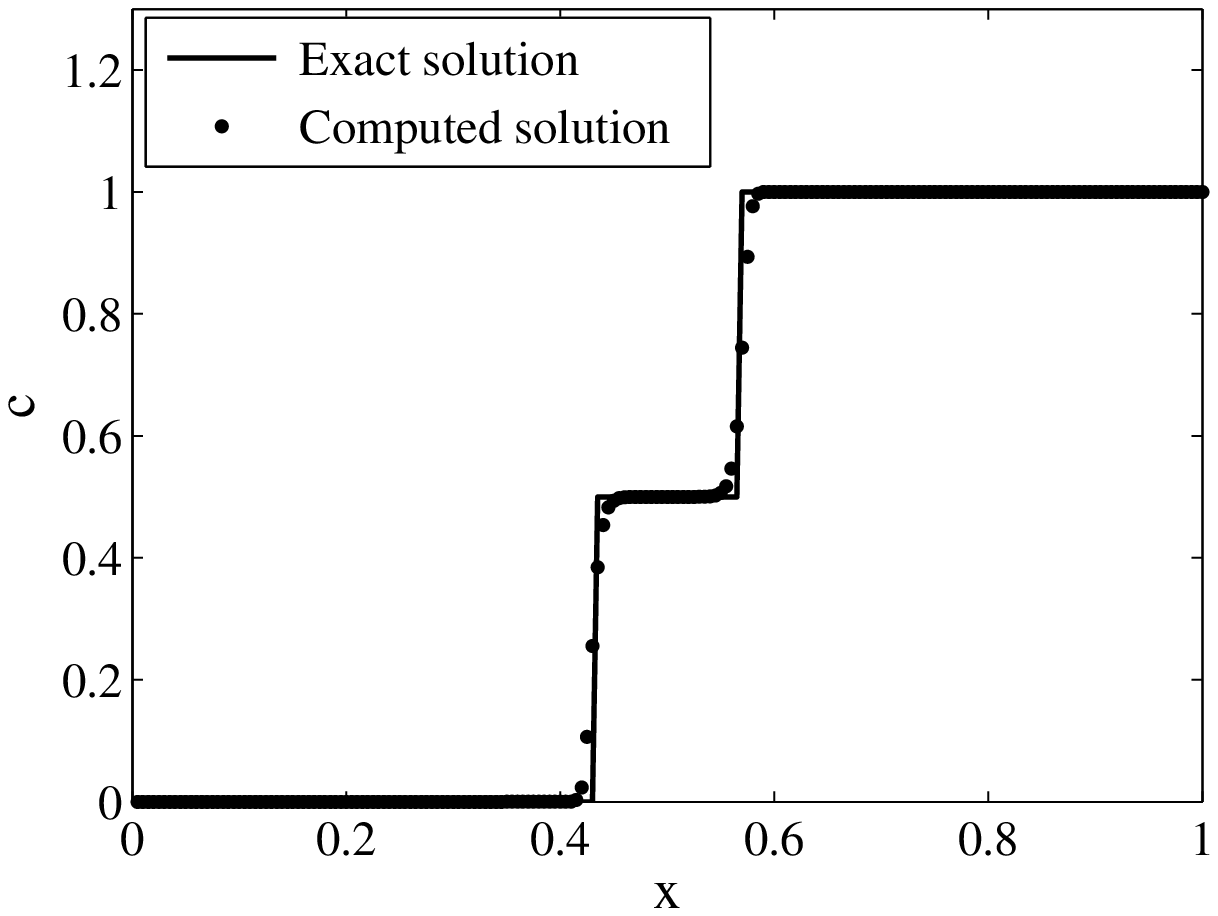}  & \includegraphics[width=6.4cm]{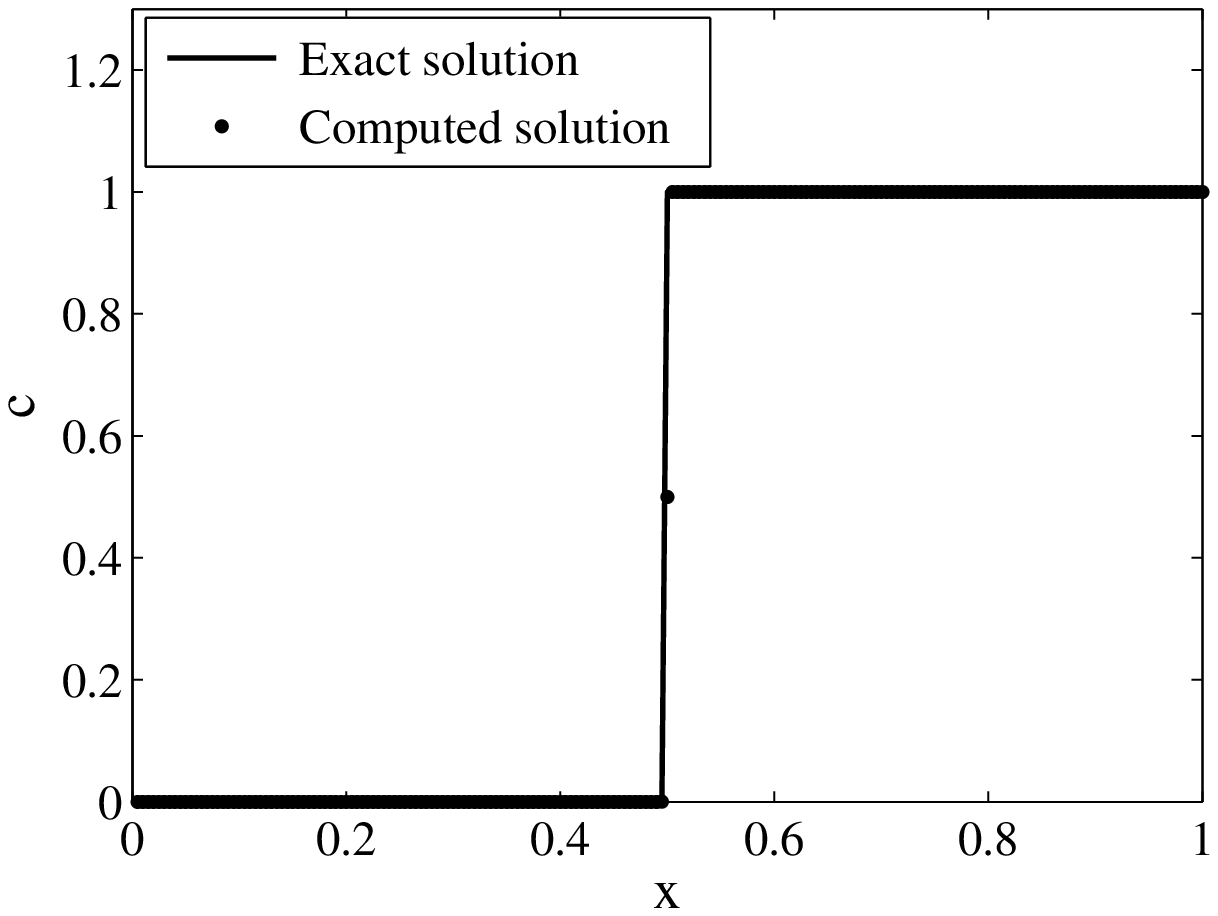} \\
\end{tabular}
\caption{Evolution of the concentration $\rs$ for the initial data defined in~\eqref{1Dim-1}, with the ``wall'' boundary condition~\eqref{1Dim-1a}. Exact solution (plain) and  discrete solution (dotted line). Simulations performed for $N_x = 200$.}\label{fig:1Dim-1}
\end{center}
\end{figure}

As we can see on Figures~\ref{fig:1Dim-1}b-c and  \ref{fig:1Dim-2}b-c, errors between the computed and the exact solution happen mostly near the points of discontinuity of the exact solution, except when we reach steady state. 
 Discontinuities of steady states are indeed captured with striking accuracy (see Figs~\ref{fig:1Dim-1}d and \ref{fig:1Dim-2}d), whereas upwinding might be expected to diffuse fronts: it is due to the behavior of the model we consider, which tends to sweep  $\rho$ toward the discontinuity, inducing permanent correction of the numerical diffusion.


\begin{figure}
\begin{center}
\begin{tabular}{ll}
\textbf{a)}  & \textbf{b)}  \\
  \includegraphics[width=6.4cm]{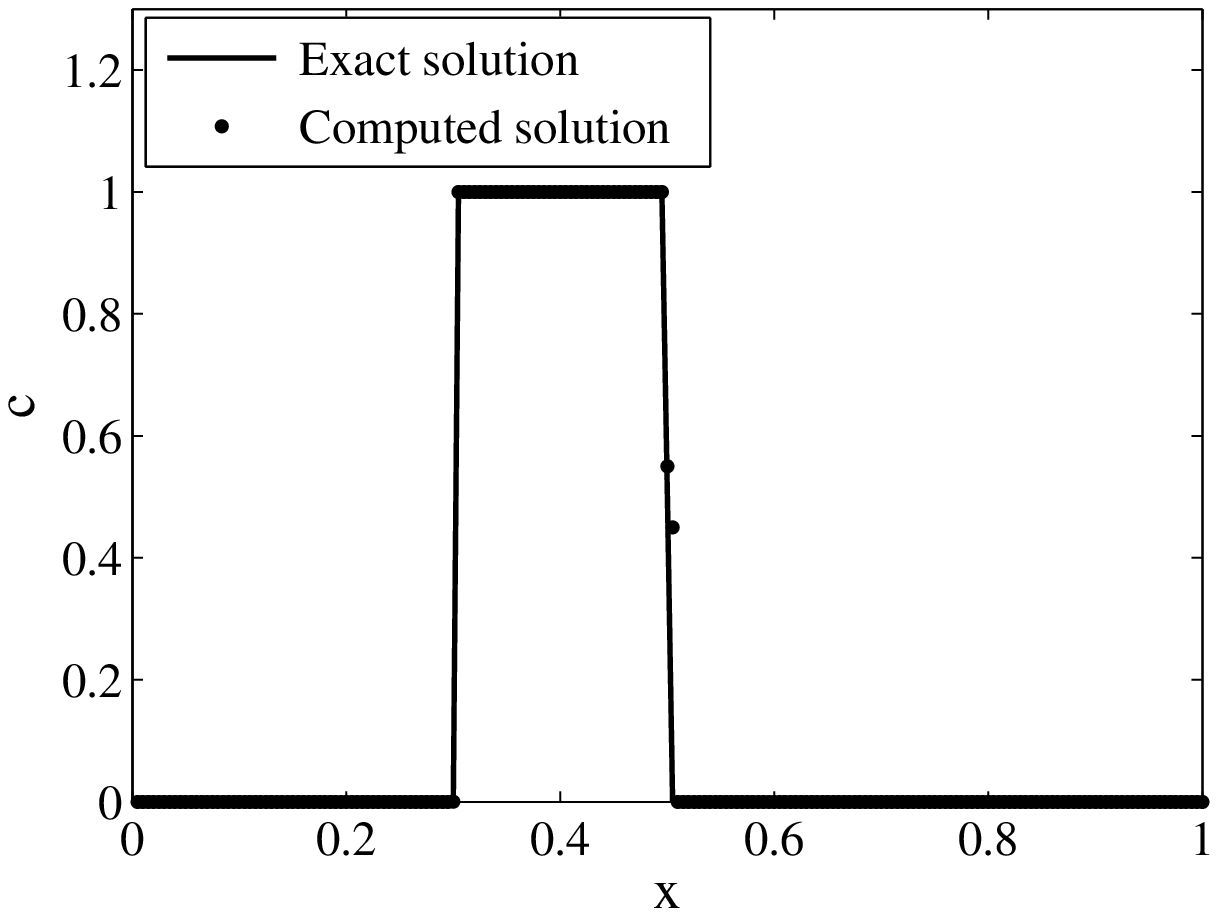}  & \includegraphics[width=6.4cm]{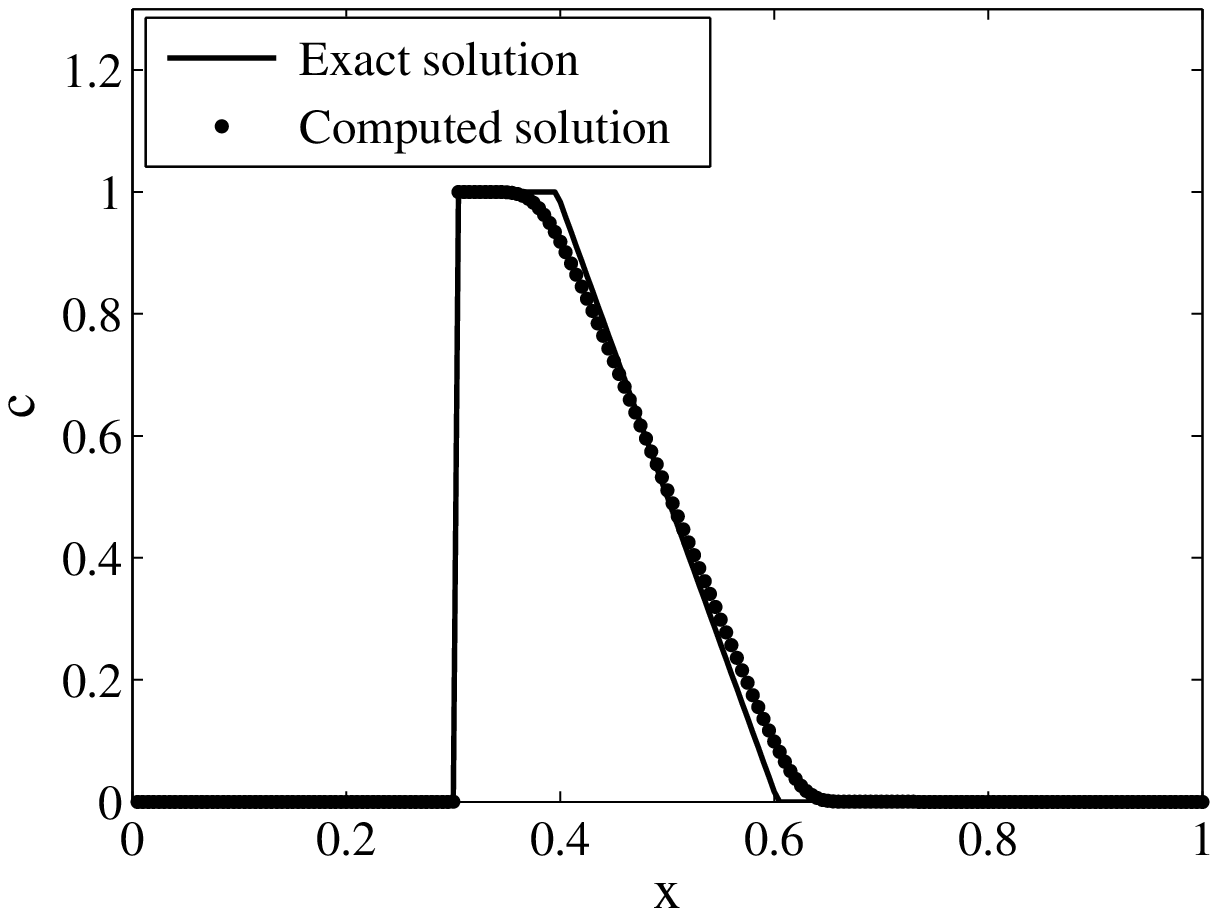} \\
  \textbf{c)}  & \textbf{d)}  \\
  \includegraphics[width=6.4cm]{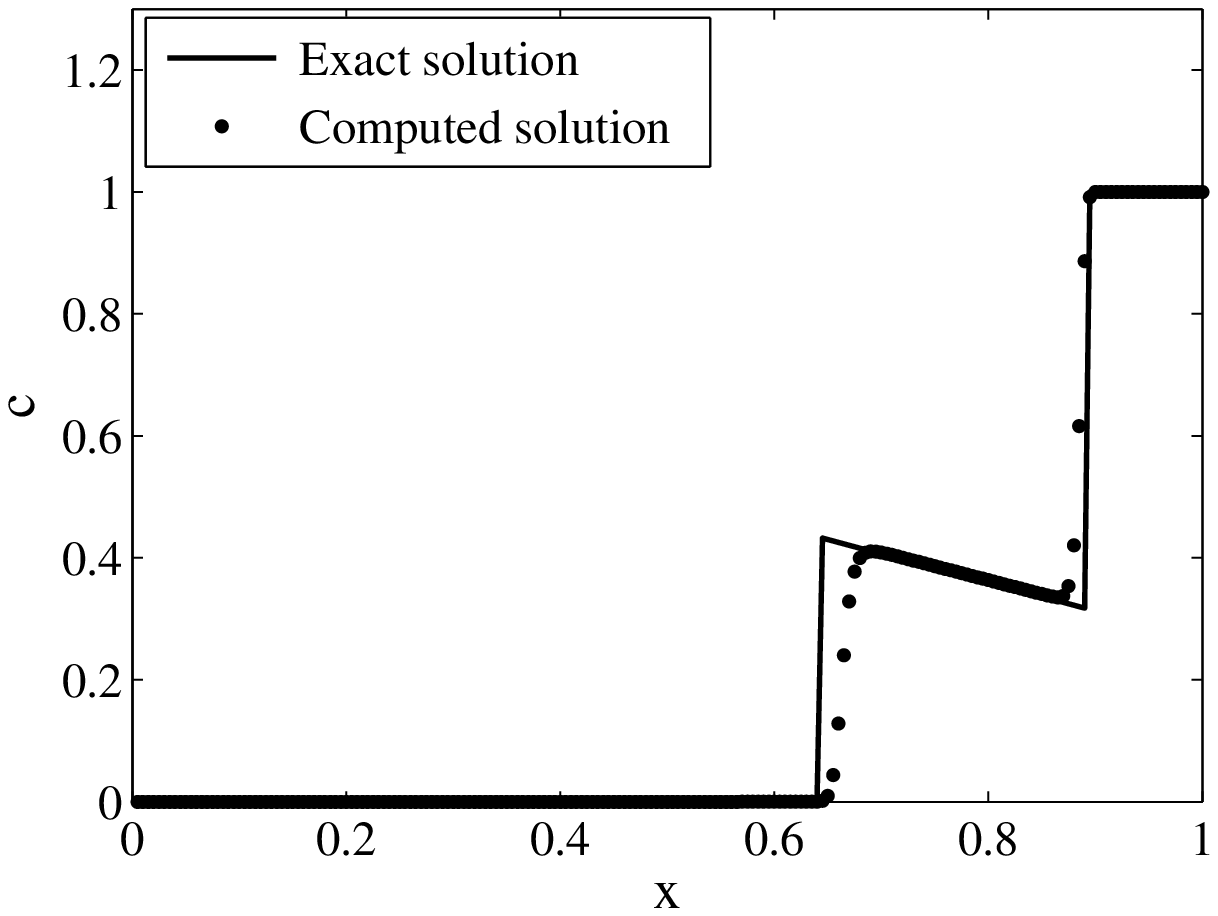}  & \includegraphics[width=6.4cm]{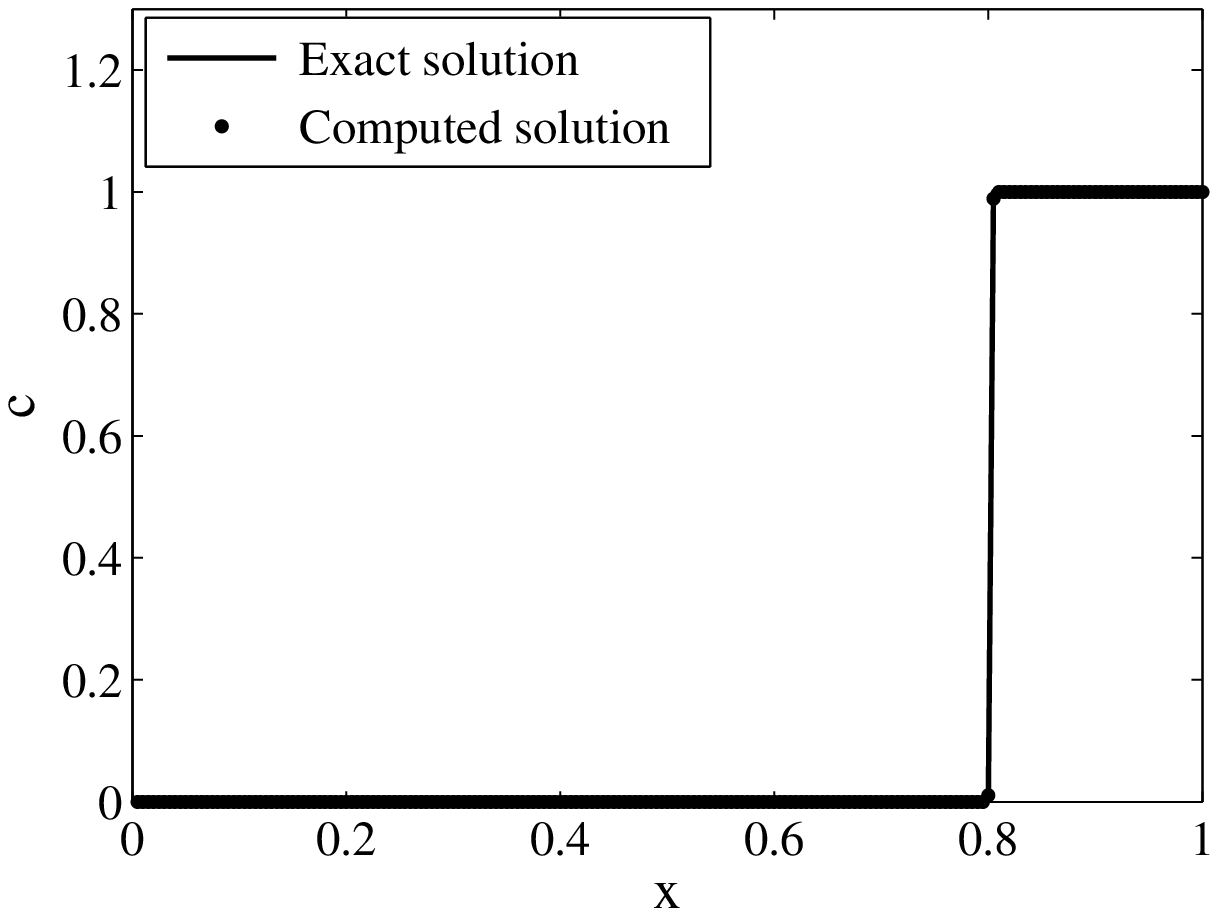} \\
\end{tabular}
\caption{Evolution of the concentration $\rs$ for the initial data defined in~\eqref{1Dim-2}, with the ``wall'' boundary condition~\eqref{1Dim-1a}. Exact solution (plain) and discrete solution (dotted line).
Simulations performed for $N_x = 200$.}
\label{fig:1Dim-2}
\end{center}
\end{figure}

\subsubsection{2D simulations for a constant desired velocity $\V$}

In this section we perform some numerical simulations for the migration model~\eqref{Num-1}-\eqref{Num-3}, in the  case  where the desired velocity $\V$ is given and constant in time. In all the cases shown here, the domain $\Omega$ is bounded and surrounded by walls, so that $p$ verifies Neumann boundary conditions:
 \begin{equation}
\frac{\partial p }{\partial n }= \rs \, \V \cdot\nn . \label{2Dim-1}
 \end{equation} 
The first situation studied here is a direct 2D extension of the 1D test case described previously, see Fig.~\ref{fig:1Dim-2}. The domain is the unit square 
$
 \Omega=(0\, , \, 1)\times (0 \, , \, 1) $.
 We also prescribe the following desired velocity: $\V =1$. Figure~\ref{fig:Mig-1} shows the evolution of $\rs$ from the initial condition (Fig.~\ref{fig:Mig-1}a) to the steady state (Fig.~\ref{fig:Mig-1}f) given by a numerical simulation performed in the situation described above. The same way as in the 1D situation (Fig.~\ref{fig:1Dim-2}b), the evolution of $\rs$ begins with the spreading of a mixing zone along the $x$ axis (Figs. \ref{fig:Mig-1}b-c). Along with the evolution of the mixing zone, the solution also tends to spread along the $y$ axis. Hence, a real difference is shown here between our model and Burgers-type models. When the mixing zone reaches the wall, a congested zone instantly appears and develops backwards (Fig.~\ref{fig:Mig-1}d-e). The steady state is reached when all the right side of the domain is saturated with the active species (Fig.~\ref{fig:Mig-1}f). \par
\begin{figure}
\begin{center}
\begin{tabular}{ll}
\textbf{a)}  & \textbf{b)}  \\
  \includegraphics[width=6.4cm]{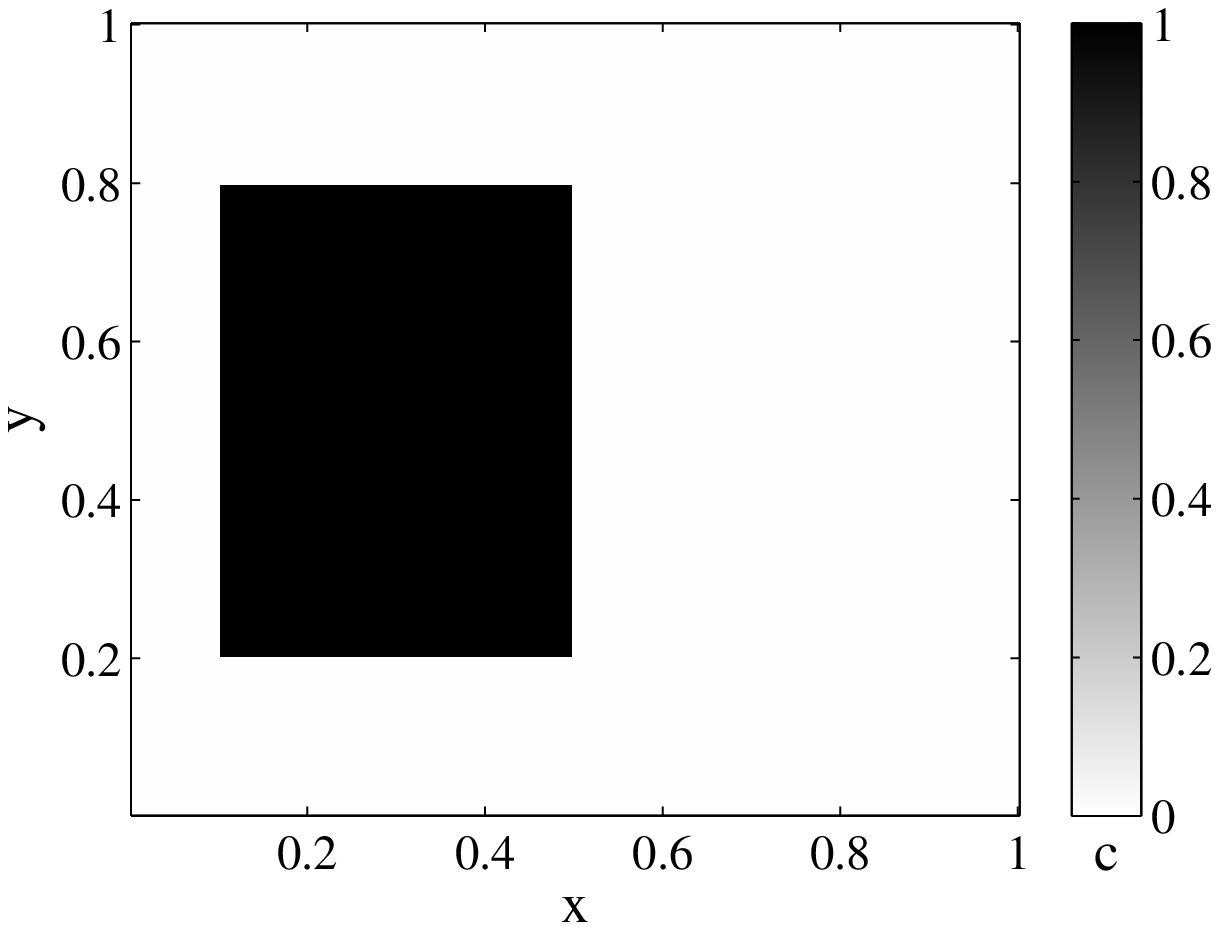}  & \includegraphics[width=6.4cm]{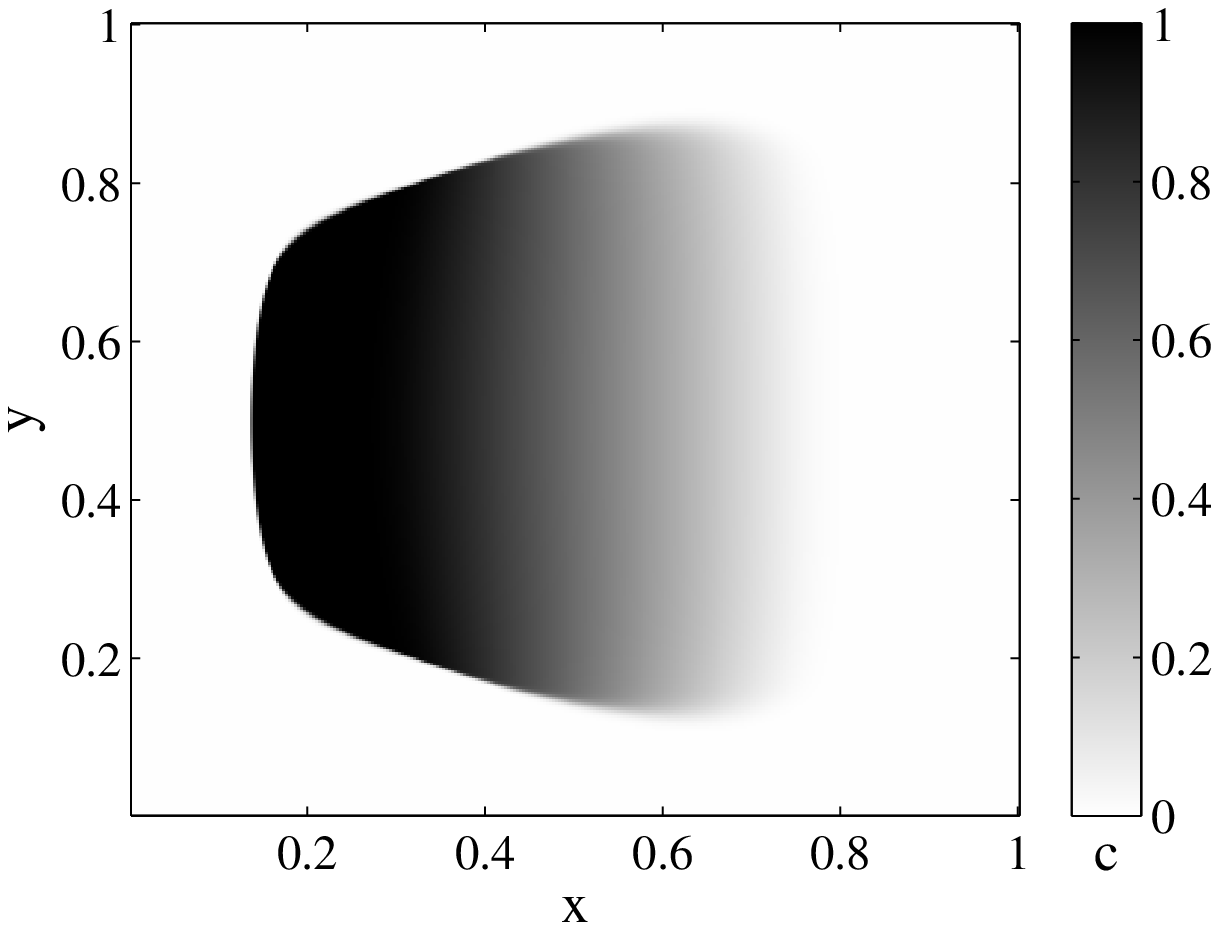} \\
 \textbf{c)}  & \textbf{d)}  \\
  \includegraphics[width=6.4cm]{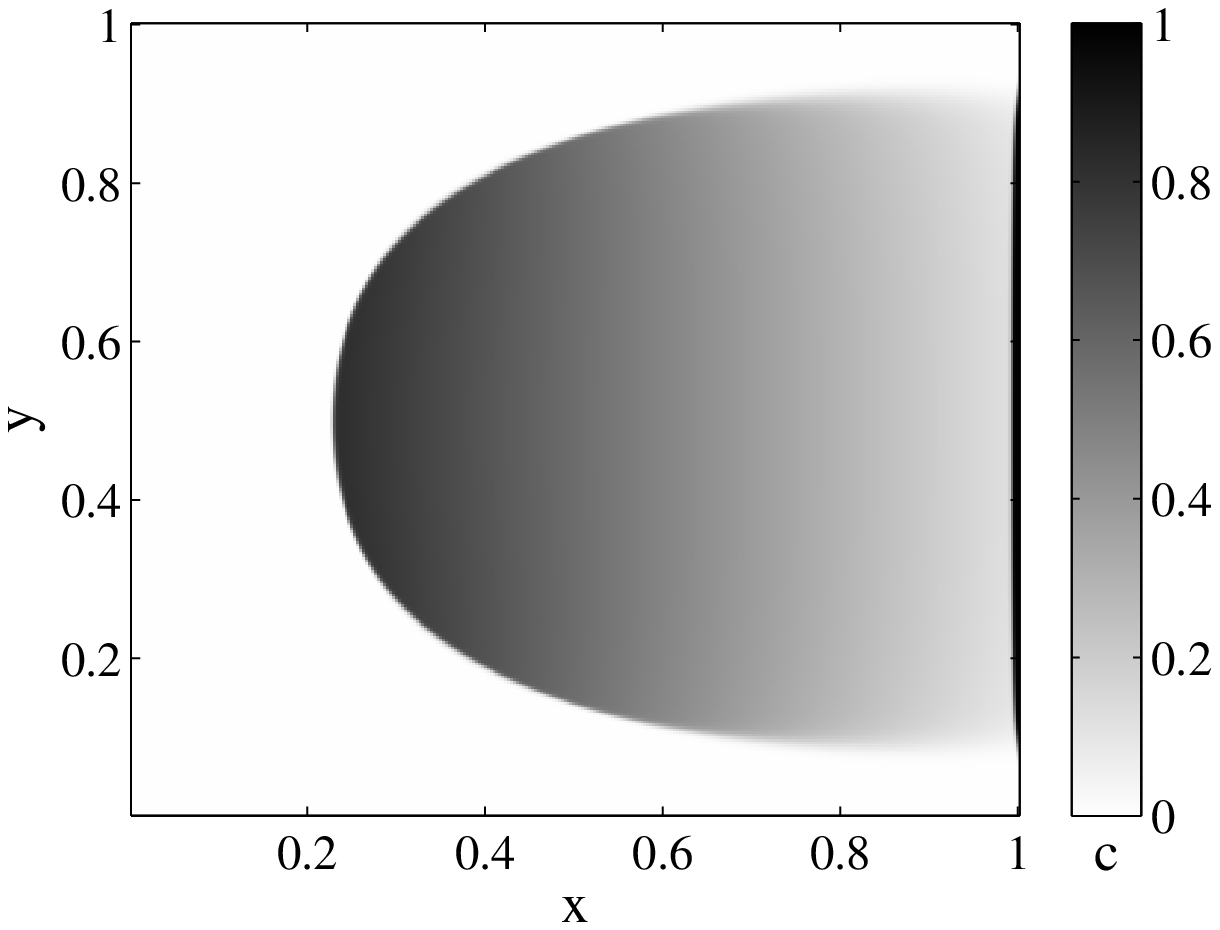}  & \includegraphics[width=6.4cm]{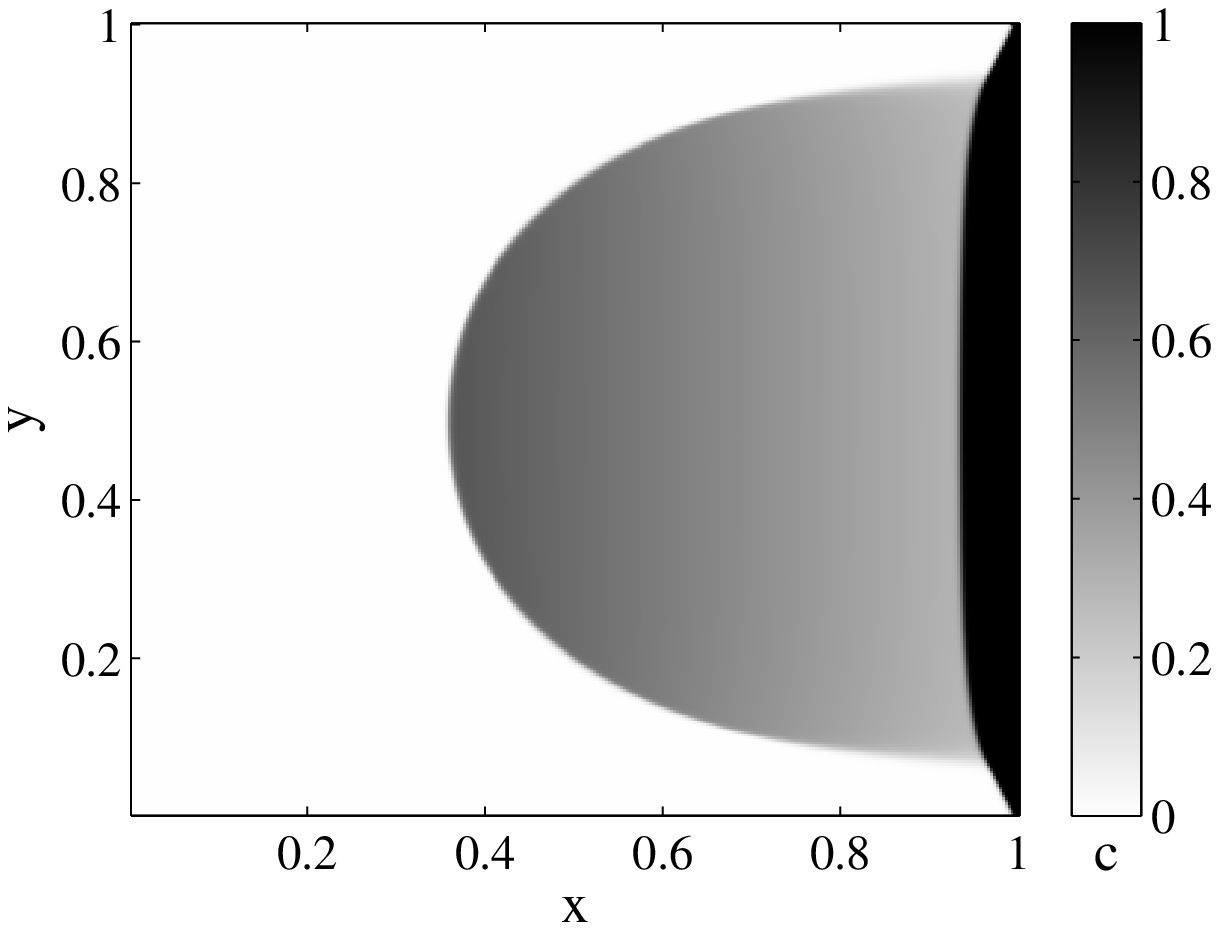} \\
  \textbf{e)}  & \textbf{f)}  \\
  \includegraphics[width=6.4cm]{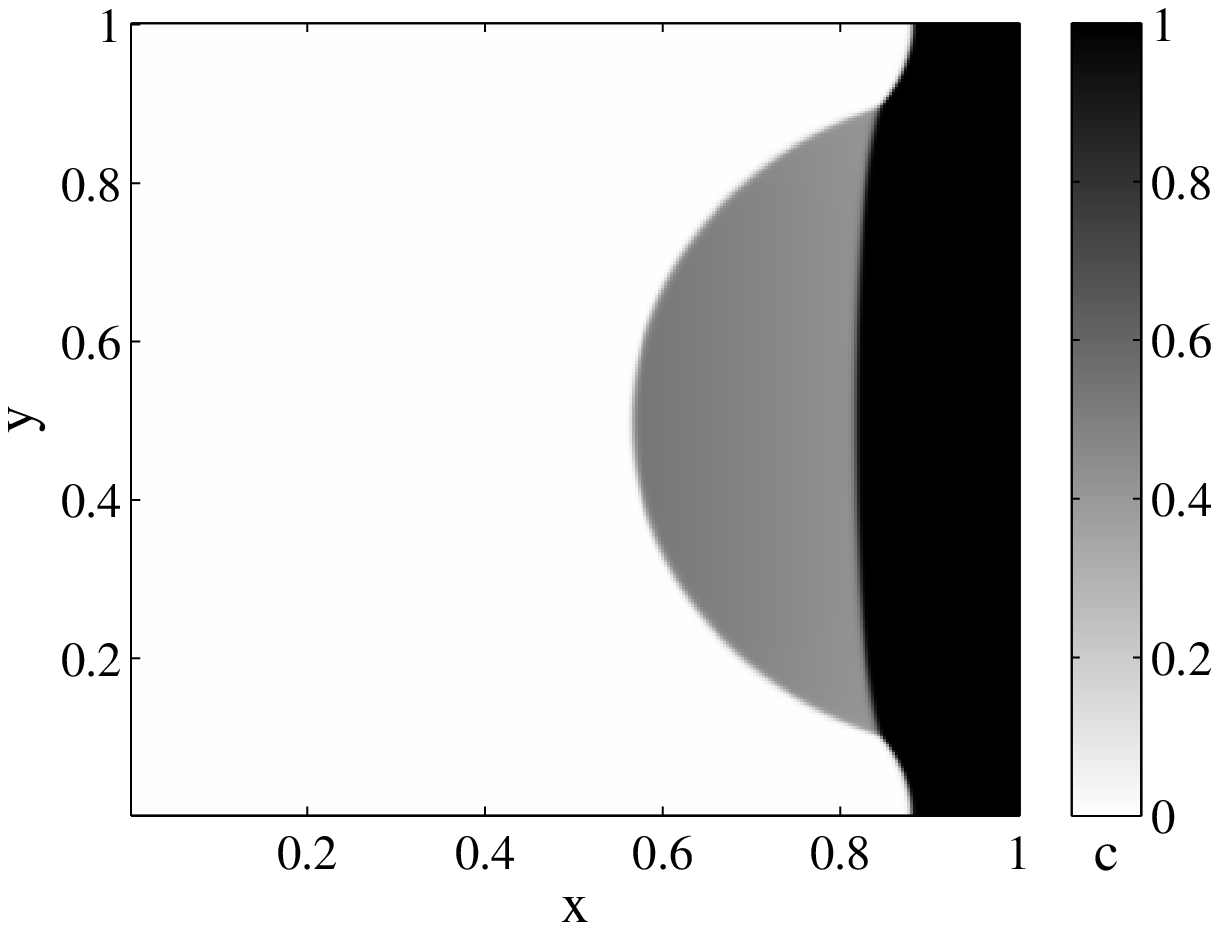}  & \includegraphics[width=6.4cm]{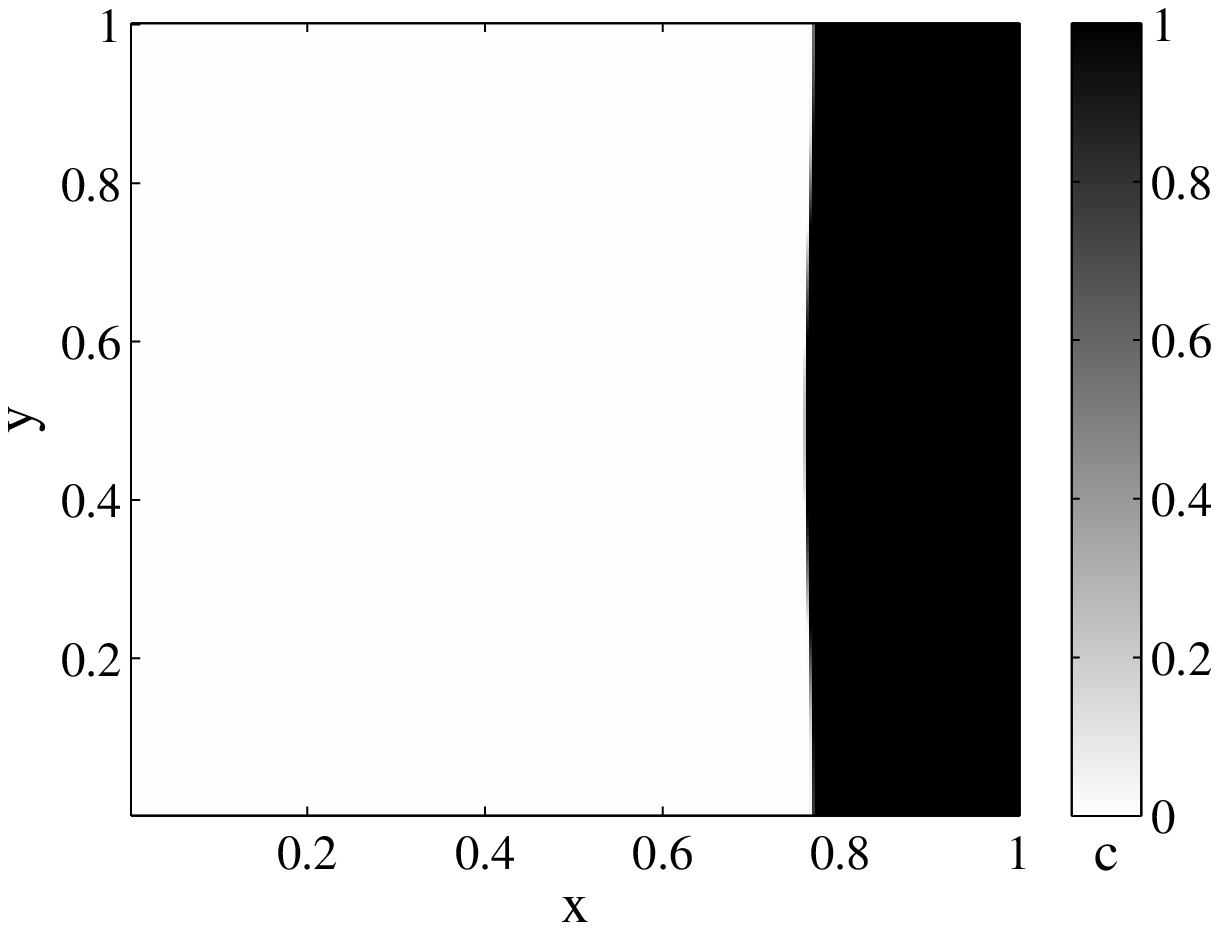} \\
\end{tabular}
\caption{Evolution of the concentration $\rs$ in the case of a constant desired velocity $\V=1$, with the ``wall'' boundary condition~\eqref{2Dim-1}. Mesh resolution: $N_x \times N_y = 300 \times 300$.}\label{fig:Mig-1}
\end{center}
\end{figure}
The second situation studied in this section involves a more complex geometry for the domain, which is defined by
\begin{equation}
\Omega = (0\, , \, 1)\times (0 \, , \, 1) \backslash \left(  [0.3\, , \, 0.7]\times [0 \, , \, 0.45] \cup [0.3\, , \, 0.7]\times [0.55 \, , \, 1] \right)  \, , 
\end{equation}
which corresponds to two rectangular rooms joined by a thin corridor (see Fig.~\ref{fig:Mig-2}). The desired velocity is set to $\V=- \nabla D$, where $D(x,y)$ represents the geodesic distance from $(x,y)$ to the right wall. In practice $D$ is computed thanks to the toolbox provided by~\cite{Peyre-2008}, which is based on the Fast-Marching method (see~\cite{KS}).
\begin{figure}
\begin{center}
\includegraphics[width=6.4cm]{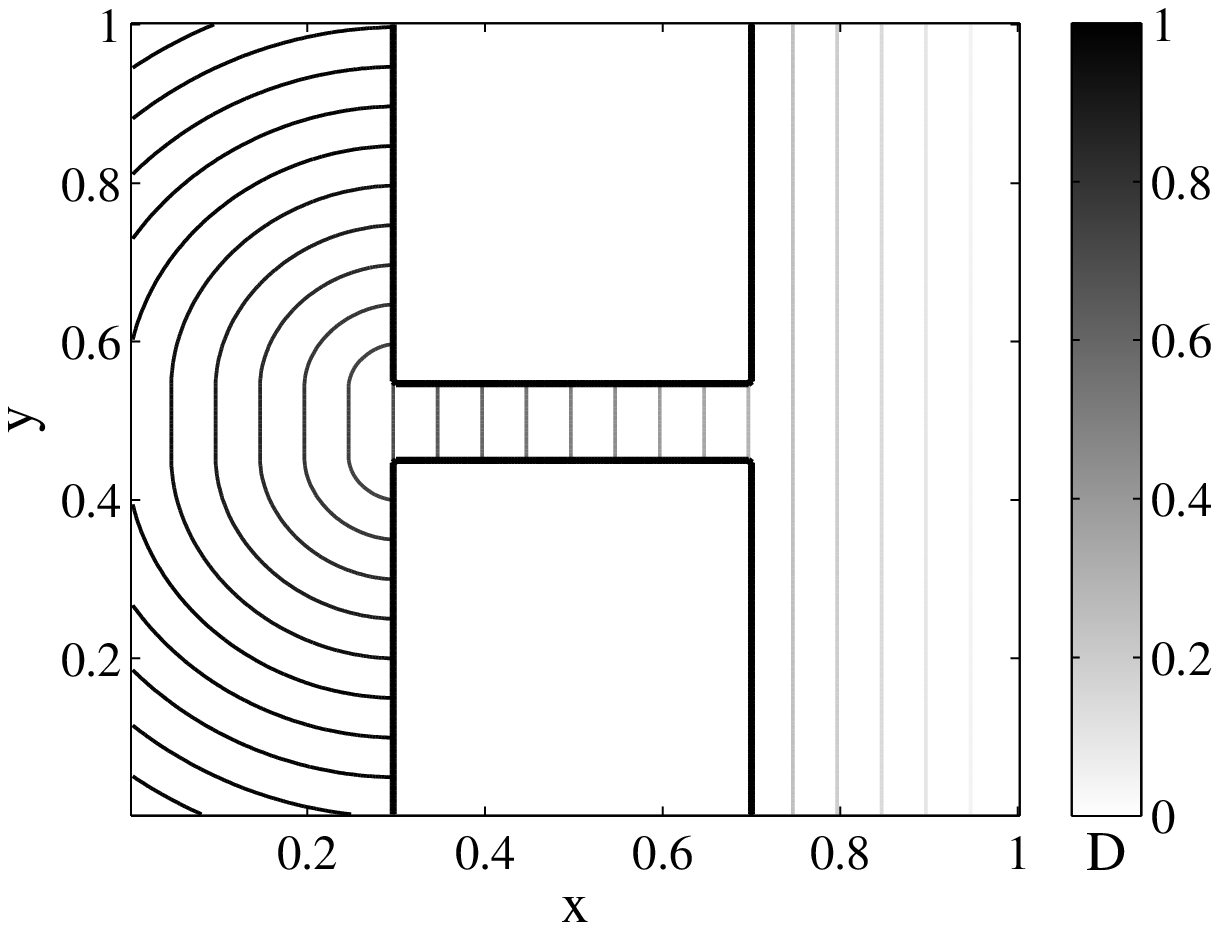}
\caption{Geodesic distance $D(x,y)$ to the right wall taking into account the obstacle computed with a fast marching algorithm.
}\label{fig:Mig-2}
\end{center}
\end{figure}
Figure \ref{fig:Mig-3} represents the simulated evolution of the concentration $\rs$ with the setting described above. The initial condition is shown in figure \ref{fig:Mig-3}a. As we can see a congested zone rapidly forms near the entrance of the corridor (figure~\ref{fig:Mig-3}b), which size decreases in time as the matter runs through it (figures \ref{fig:Mig-3}c-e). We notice that a quasi-homogeneous flow regime tends to be established in the corridor (Fig.~\ref{fig:Mig-3}d-e), where a constant concentration $\rs\sim 0.5$ is observed. It illustrates conservation of species $2$ (in white in the figures): as the domain is bounded by walls, flow of $1$ to the right has to be balanced by an opposite flow of $2$ to the left. 
As we define the correction velocity  in the least-square sense, actual speeds  for $1$ and $2$ are close,  so that the mixture is necessarily balanced for global mass conservation reasons. 
 After some time, as observed in the previous test case, all the active species is concentrated on the right side of the domain (Fig.~\ref{fig:Mig-3}f).
\begin{figure}
\begin{center}
\begin{tabular}{ll}
\textbf{a)}  & \textbf{b)}  \\
  \includegraphics[width=6.4cm]{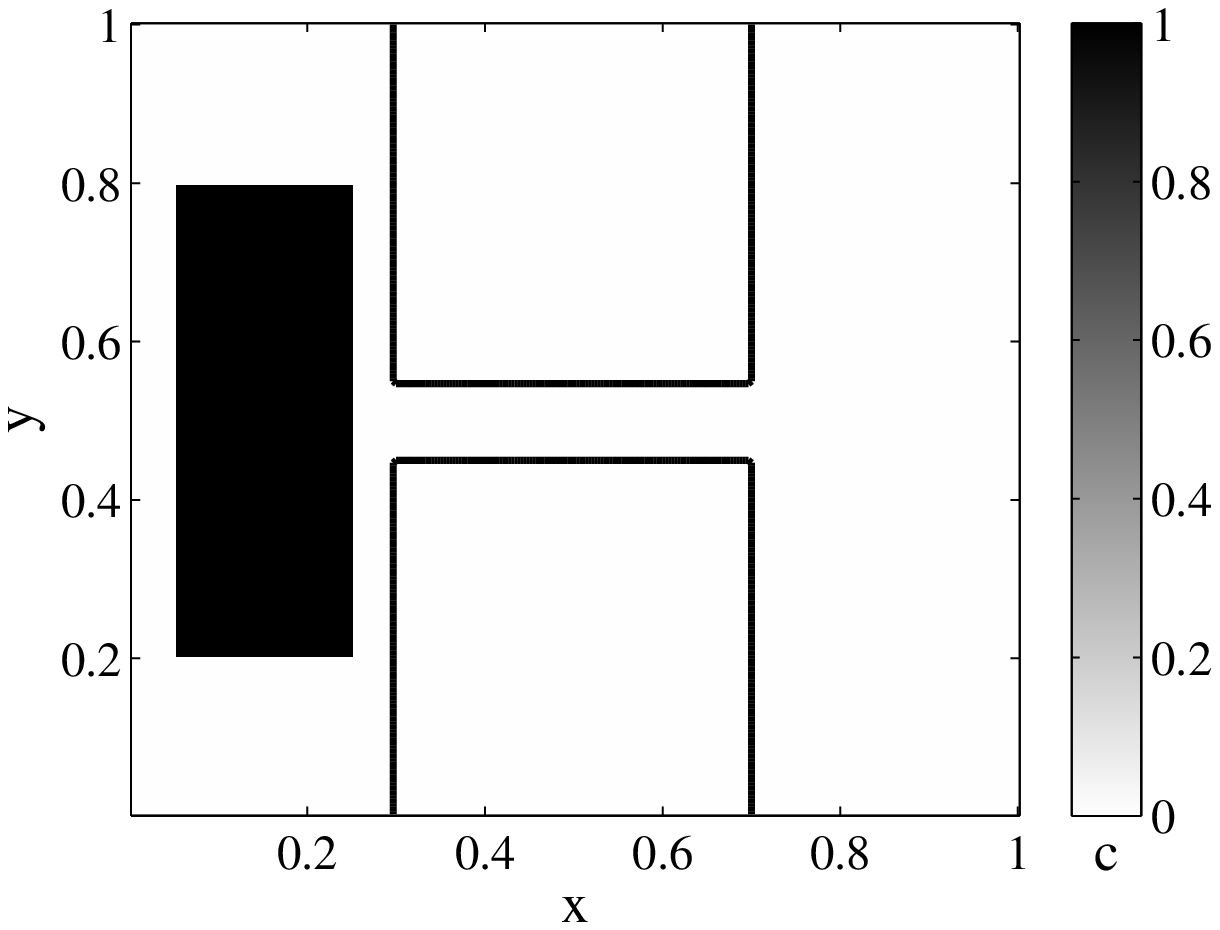}  & \includegraphics[width=6.4cm]{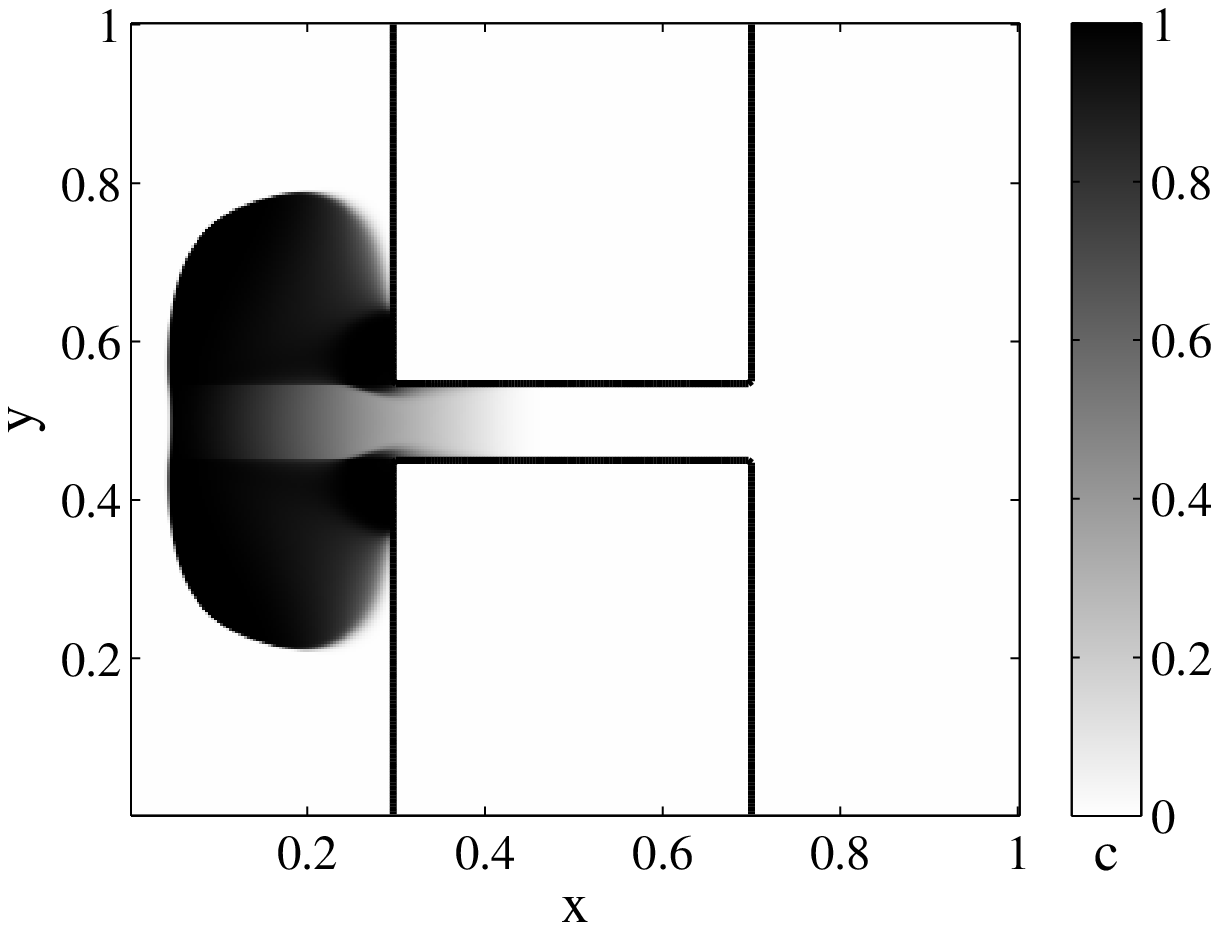} \\
 \textbf{c)}  & \textbf{d)}  \\
  \includegraphics[width=6.4cm]{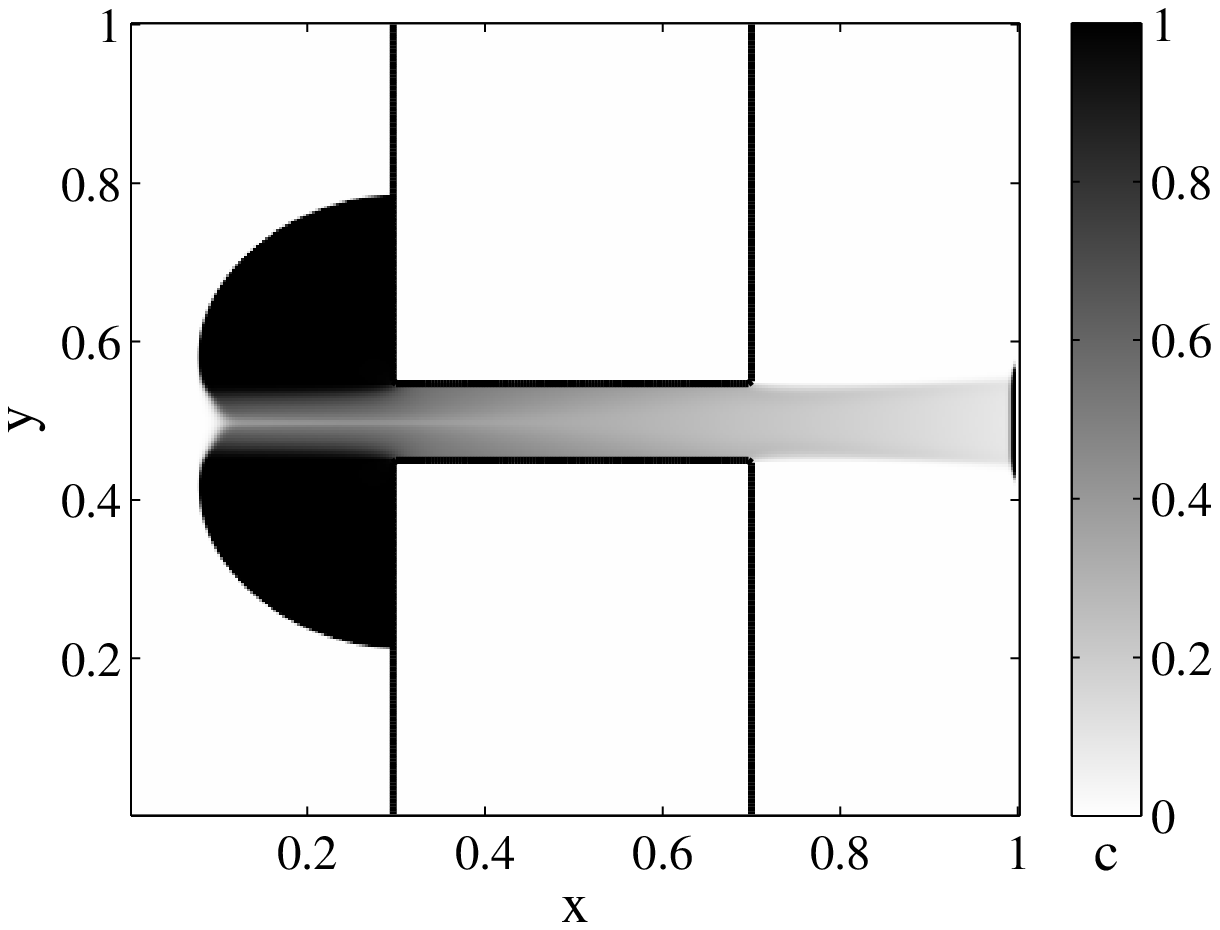}  & \includegraphics[width=6.4cm]{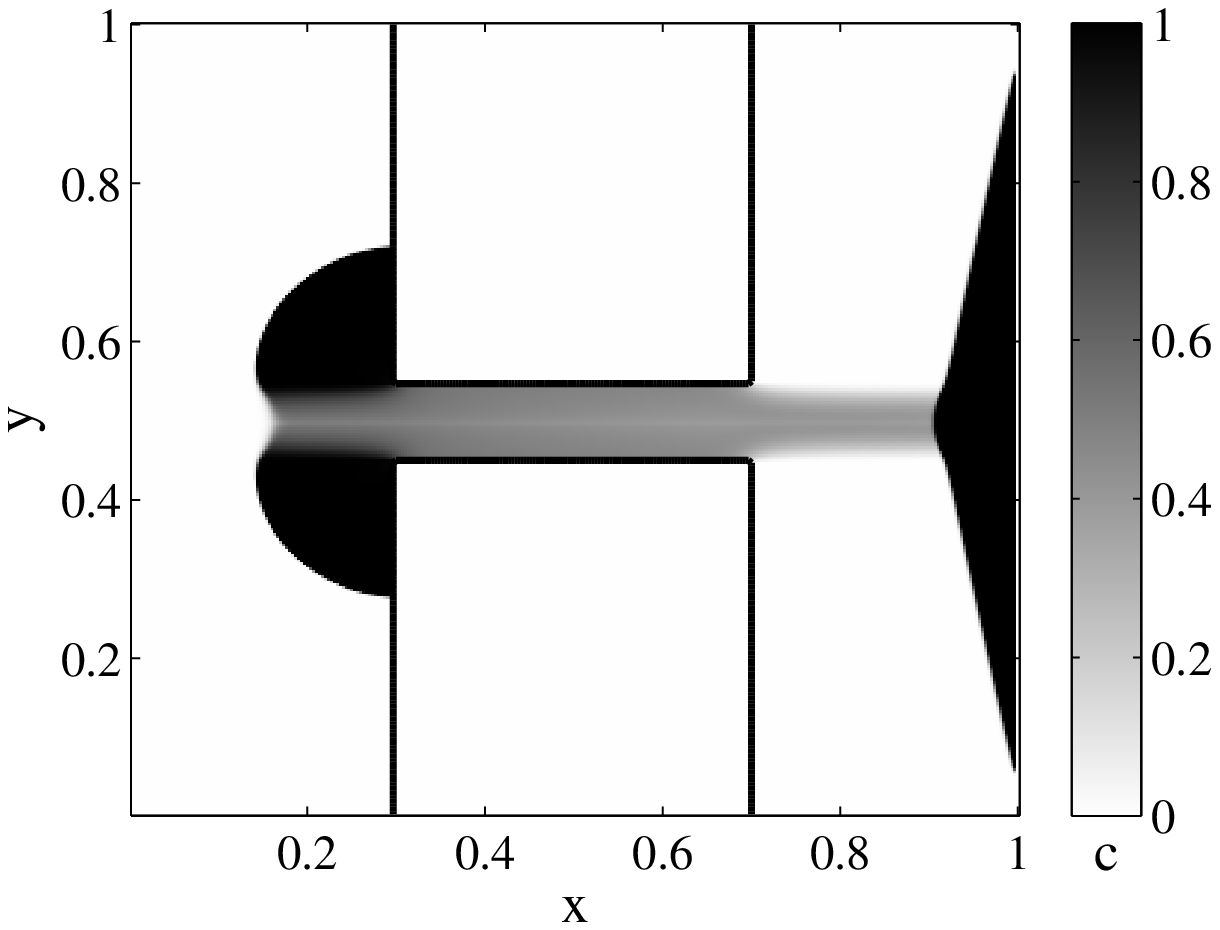} \\
  \textbf{e)}  & \textbf{f)}  \\
  \includegraphics[width=6.4cm]{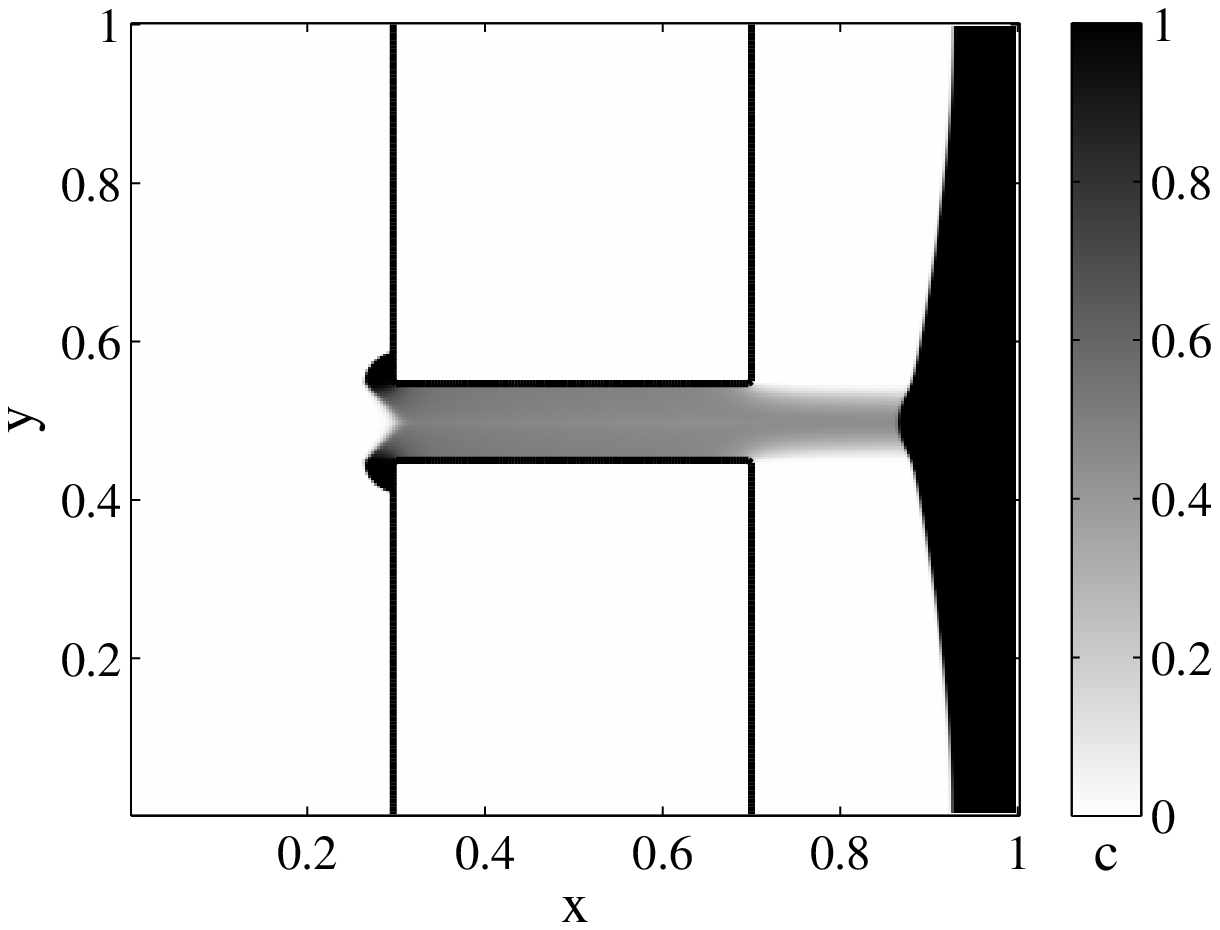}  & \includegraphics[width=6.4cm]{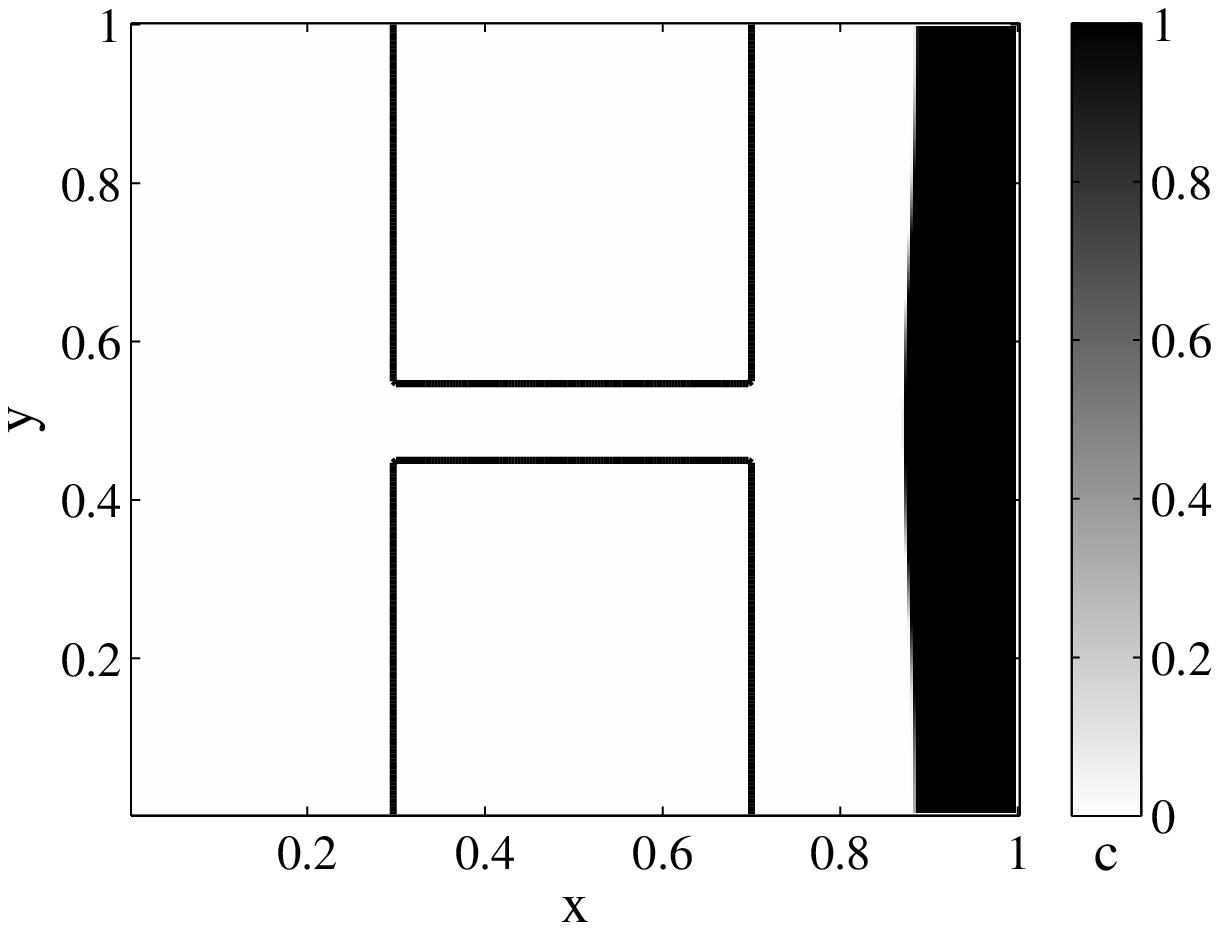} \\
\end{tabular}
\caption{Evolution of the concentration $\rs$ in the case of a constant desired velocity $\V=\nabla D$ (see Fig.~\ref{fig:Mig-2}), with the ``wall'' boundary condition~\eqref{2Dim-1}. Mesh resolution: $N_x \times N_y = 300 \times 300$.}\label{fig:Mig-3}
\end{center}
\end{figure}

\subsubsection{2D simulations for an evolving desired velocity $\V(\rs)$}
As mentionned in section \ref{sec:model}, the desired velocity can be chosen to depend on the local concentration $\rs$, the same way as in the Keller-Segel model (see \cite{KelSeg}):
\begin{align}
& \V= \nabla S \, ,  \label{KS-1} \\
& \Delta S = -\rs  \, .\label{KS-2} 
\end{align}
In this section we consider periodic boundary conditions over $(\rs,p,S)$:
\begin{align}
& \rs(t,0,y)=\rs(t,1,y)\,,  & \rs(t,x,0)=\rs(t,x,1)\, , \label{KS-2a} \\
 & p(t,0,y)=p(t,1,y)\,,  & p(t,x,0)=p(t,x,1)\, ,  \label{KS-2b} \\
& S(t,0,y)=S(t,1,y)\,,  & S(t,x,0)=S(t,x,1)\, .\label{KS-2c}
\end{align}
The numerical treatment we give to equations~\eqref{KS-1}-\eqref{KS-2} is similar to the numerical treatment of equations \eqref{Num-2}-\eqref{Num-3} described previsouly:
\begin{align}
& \V_{x, \, ij}^n =  \dfrac{S^n_{ij}-S^n_{i-1,j}}{\delta x} \, ,   \label{KS-3}  \\
& \V_{y, \, ij}^n =  \dfrac{S^n_{ij}-S^n_{i,j-1}}{\delta y} \, ,   \label{KS-4}  \\
&\dfrac{S^n_{i+1,j}-2 \,S^n_{ij} + S^n_{i-1,j} }{\delta x^2}+\dfrac{S^n_{i,j+1}-2 \,S^n_{ij} + S^n_{i,j-1} }{\delta y^2} = - \rs^n_{ij} \, .   \label{KS-5} 
\end{align}
For all the following numerical simulations the initial data is set randomly  according to Bernoulli's law
\begin{align}
&\mathcal{P}\left(\rs^0_{ij} = 0 \right )  = 1-q \, , \\
&\mathcal{P}\left ( \rs^0_{ij} = 1 \right)  = q \, ,
\end{align}
where $q\in (0,1)$  conditions the initial mass of active cells. The domain is the unit square, the boundary conditions set here are periodic boundary conditions for $\rs$, $p$ and $S$. Figures~\ref{fig:KS-1} and \ref{fig:KS-2} represent the evolution of $\rs$ given by numerical simulations upon our model with the desired velocity defined by~\eqref{KS-1}-\eqref{KS-2}, and for initial conditions where $q=0.1$ and $q=0.5$ (respectively represented in Figures~\ref{fig:KS-1}a and \ref{fig:KS-2}a). In each case we see aggregation of bigger and bigger structures as time goes (Figs.~\ref{fig:KS-1}b-e and \ref{fig:KS-2}b-e). Two different steady states are shown, both involving a steady, fully congested zone. In the case $q=0.1$ the steady state congested zone is a disc, whereas in the case $q=0.5$ a band has formed because of the periodic boundary conditions we chose to consider.

\begin{figure}
\begin{center}
\begin{tabular}{ll}
\textbf{a)}  & \textbf{b)}  \\
  \includegraphics[width=6.4cm]{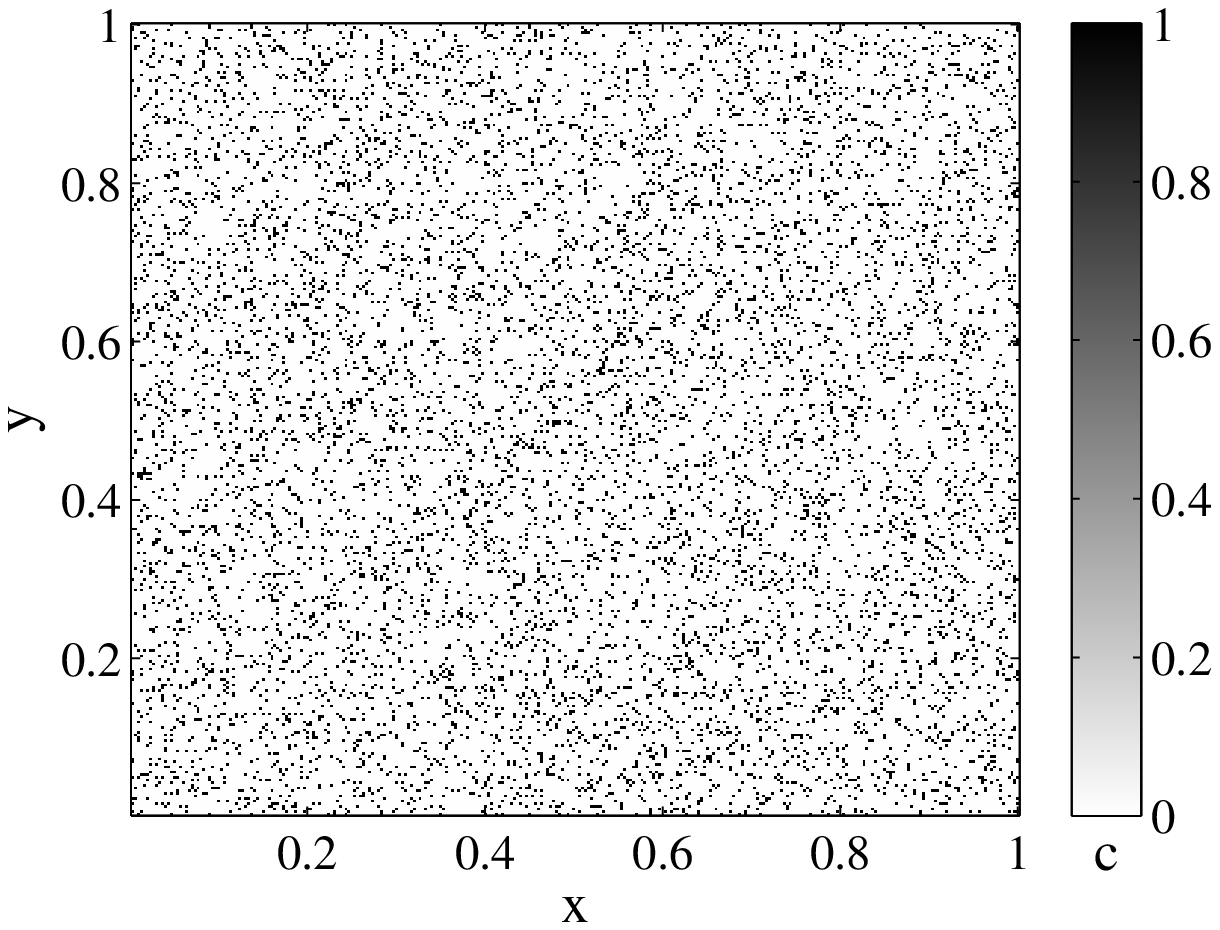}  & \includegraphics[width=6.4cm]{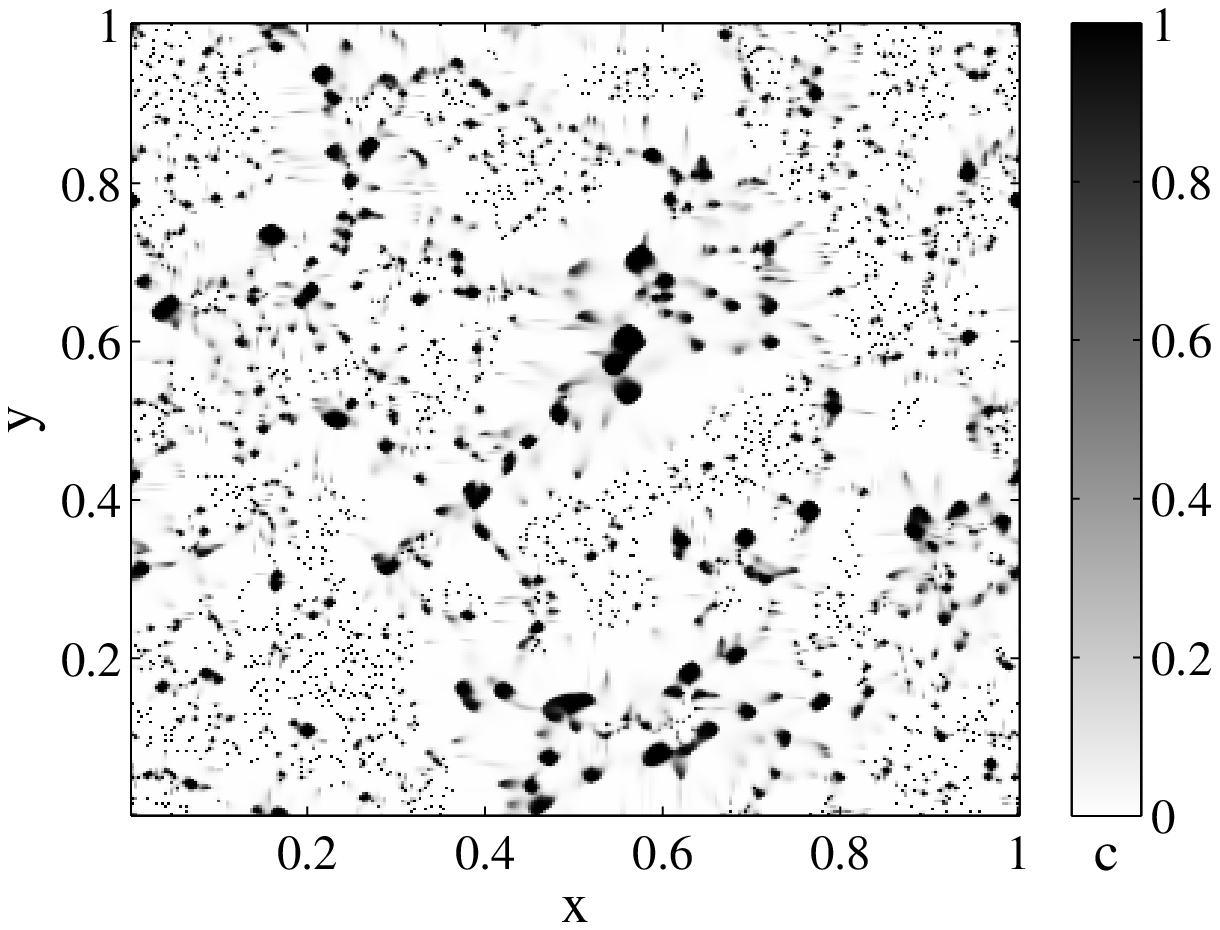} \\
 \textbf{c)} &  \textbf{d)}  \\
 \includegraphics[width=6.4cm]{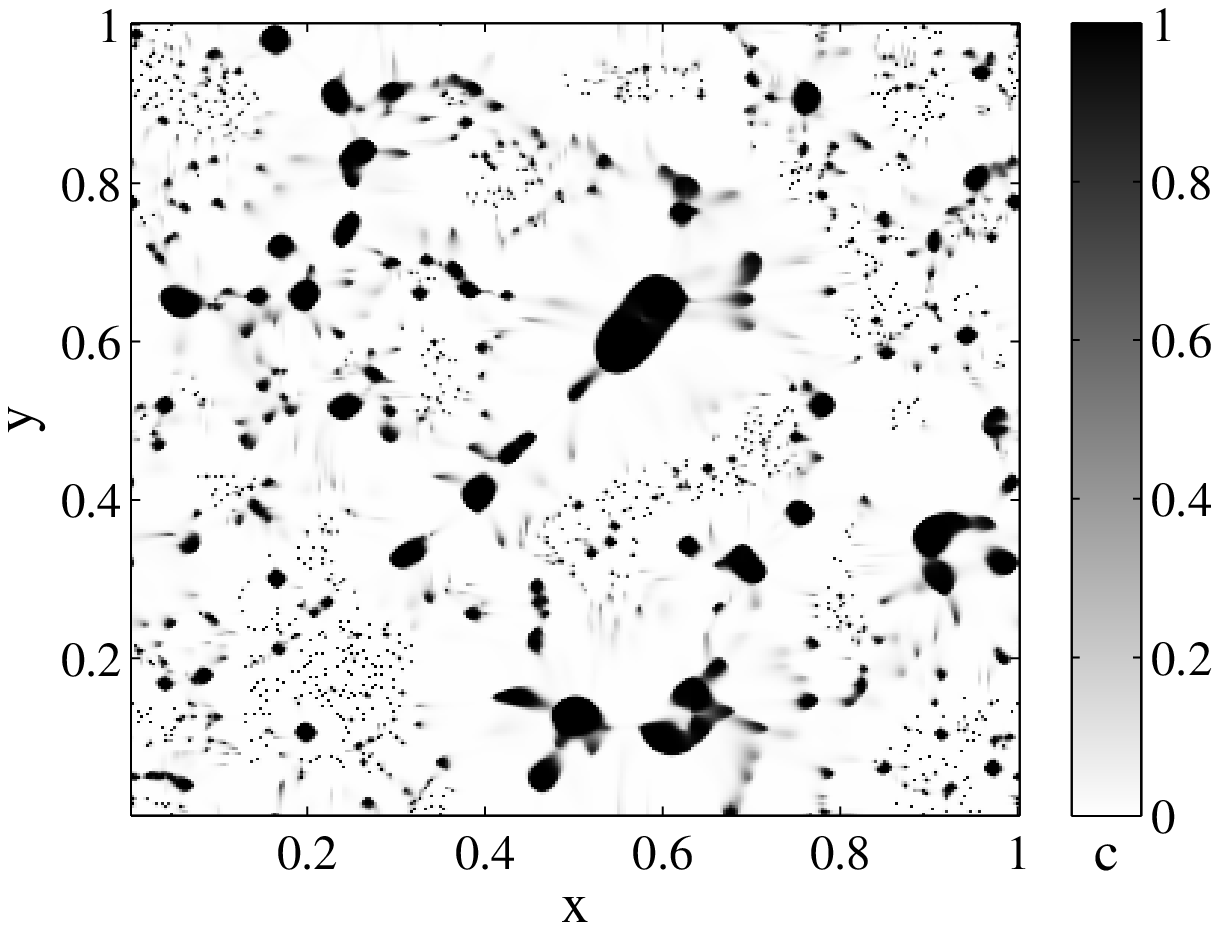}  & \includegraphics[width=6.4cm]{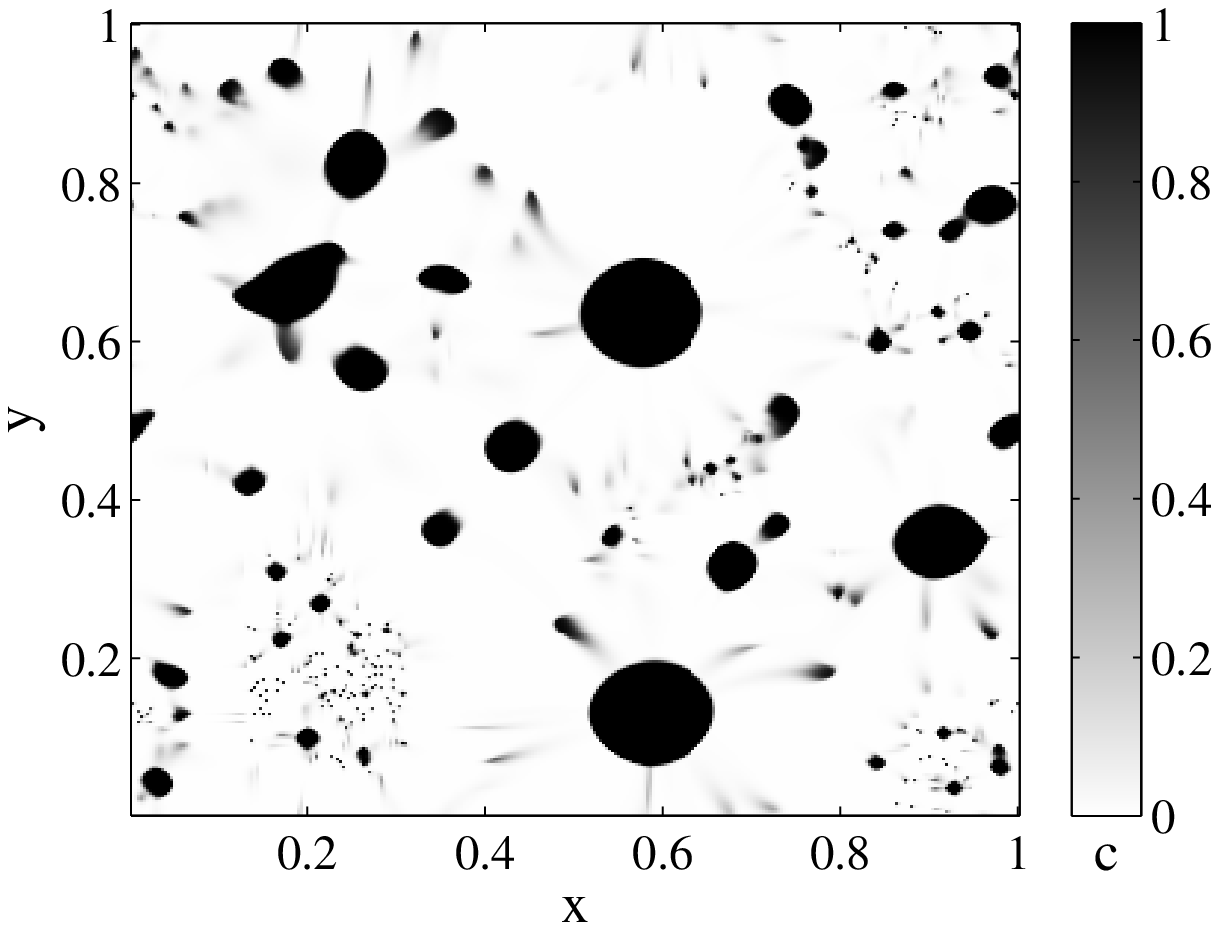} \\
  \textbf{e)} & \textbf{f)}  \\
  \includegraphics[width=6.4cm]{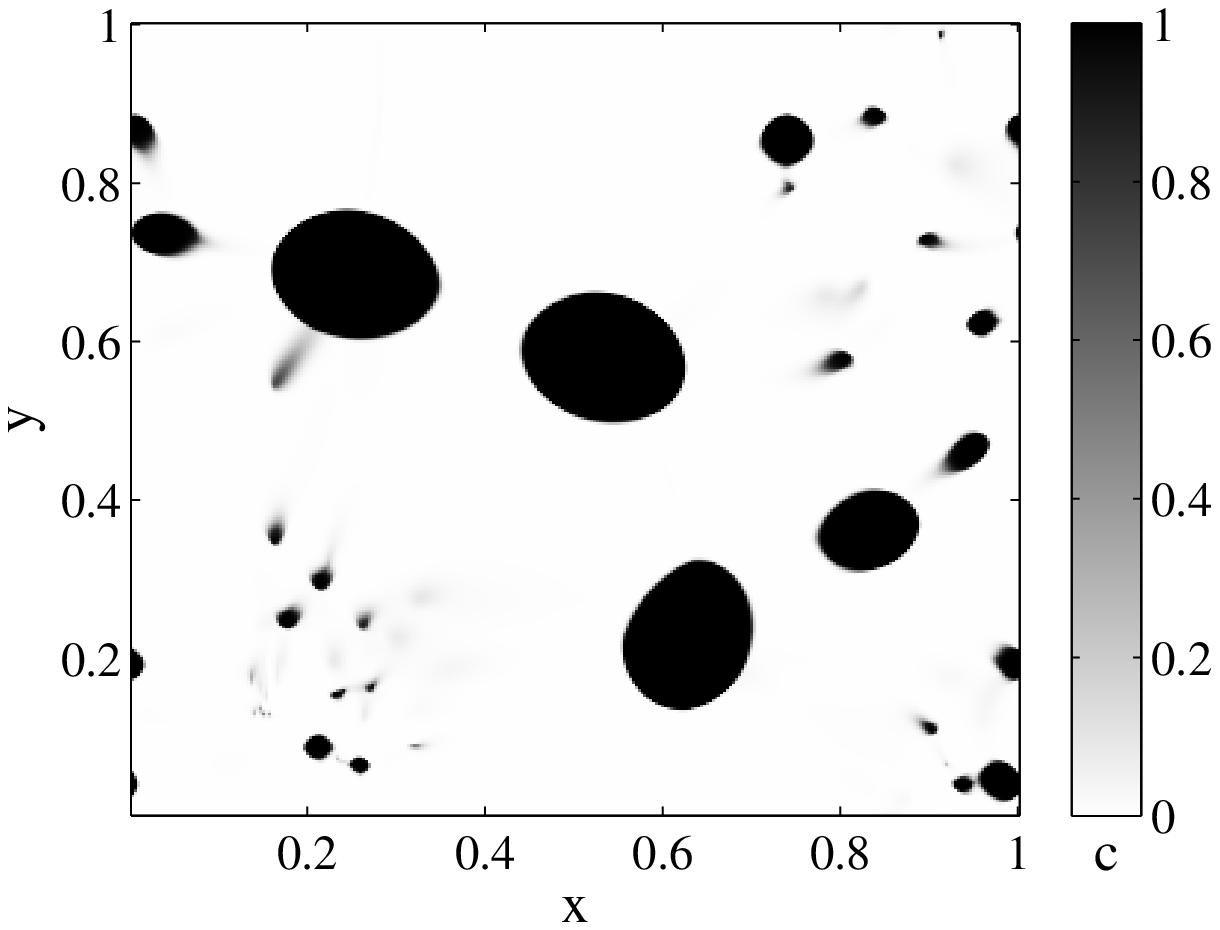} &  \includegraphics[width=6.4cm]{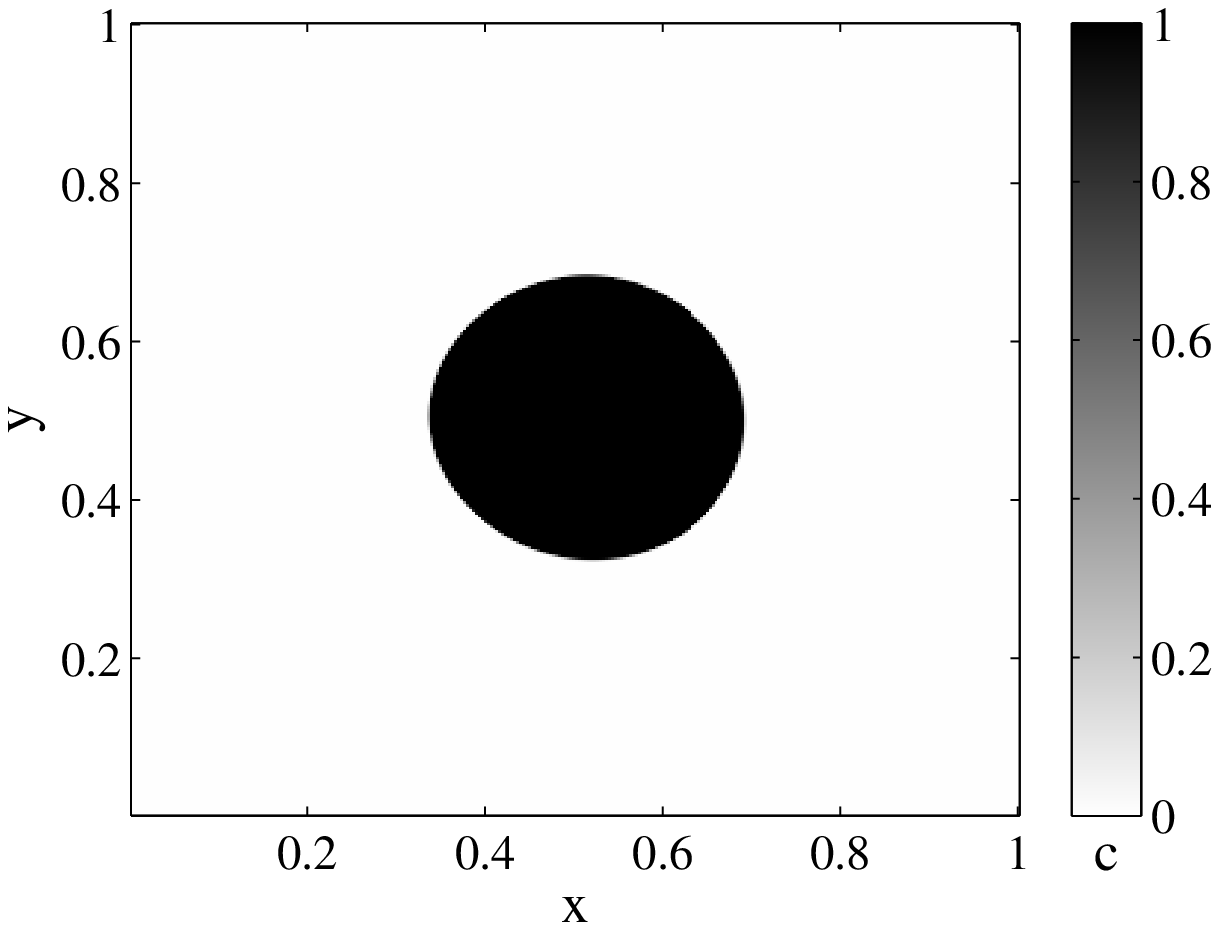}
\end{tabular}
\caption{Evolution of the concentration $\rs$ in the case of a desired velocity given by~\eqref{KS-1}-\eqref{KS-2}, with periodic  boundary conditions~\eqref{KS-2a}-\eqref{KS-2c}. The domain is initially filled randomly with 10\% moving species. Mesh resolution: $N_x \times N_y = 300 \times 300$.}\label{fig:KS-1}
\end{center}
\end{figure}

\begin{figure}
\begin{center}
\begin{tabular}{ll}
\textbf{a)} & \textbf{b)}  \\
  \includegraphics[width=6.4cm]{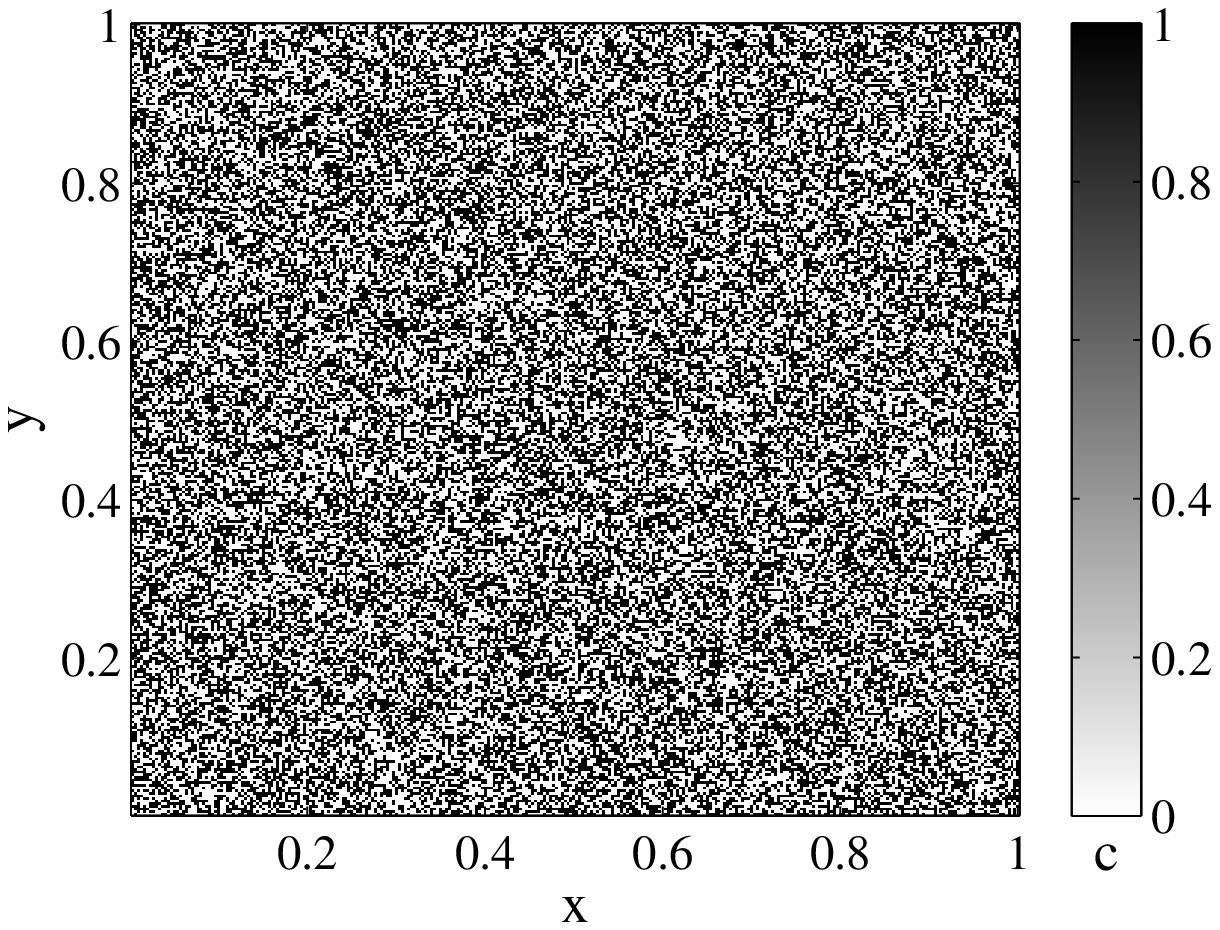} & \includegraphics[width=6.4cm]{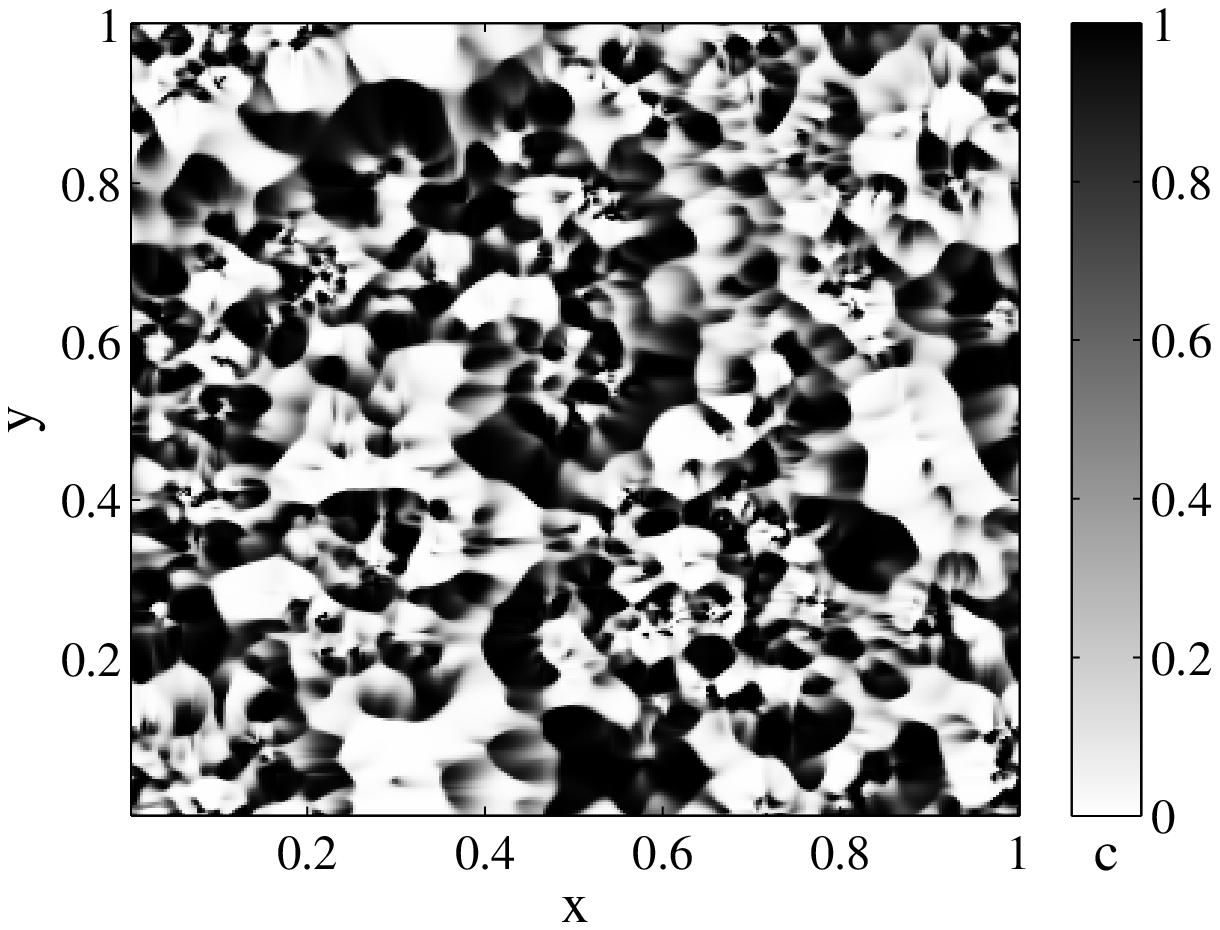} \\
 \textbf{c)} &  \textbf{d)}  \\
  \includegraphics[width=6.4cm]{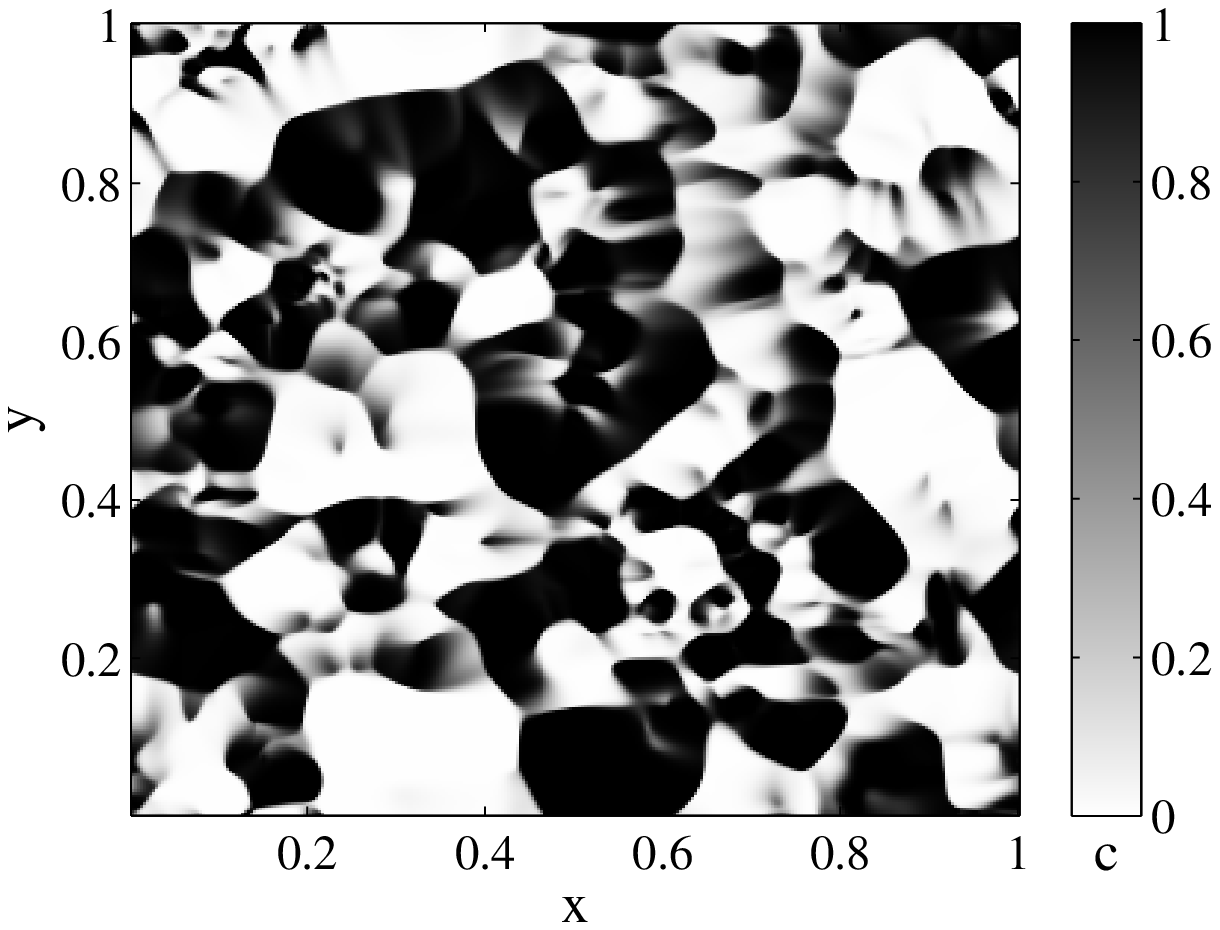} &  \includegraphics[width=6.4cm]{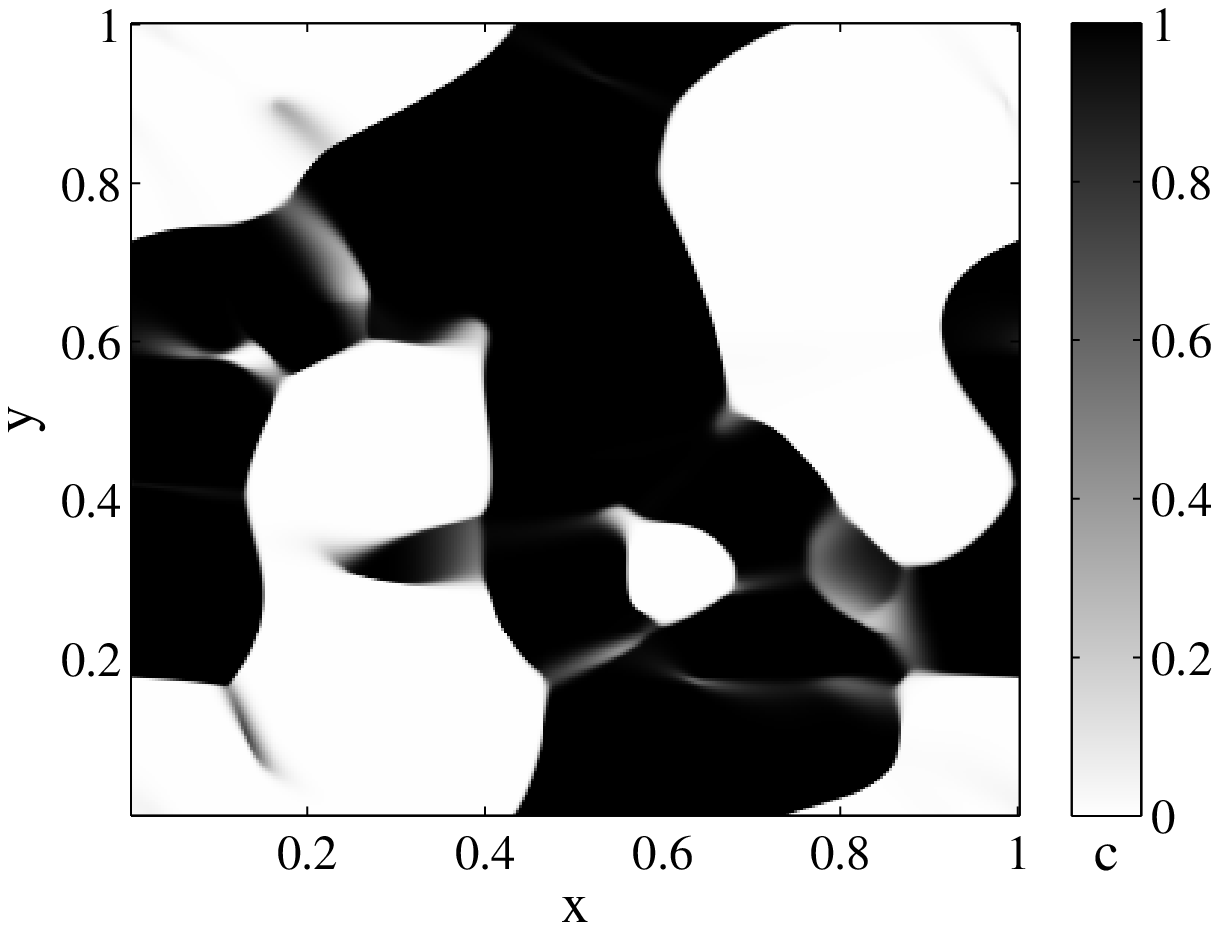} \\
  \textbf{e)} &  \textbf{f)}  \\
  \includegraphics[width=6.4cm]{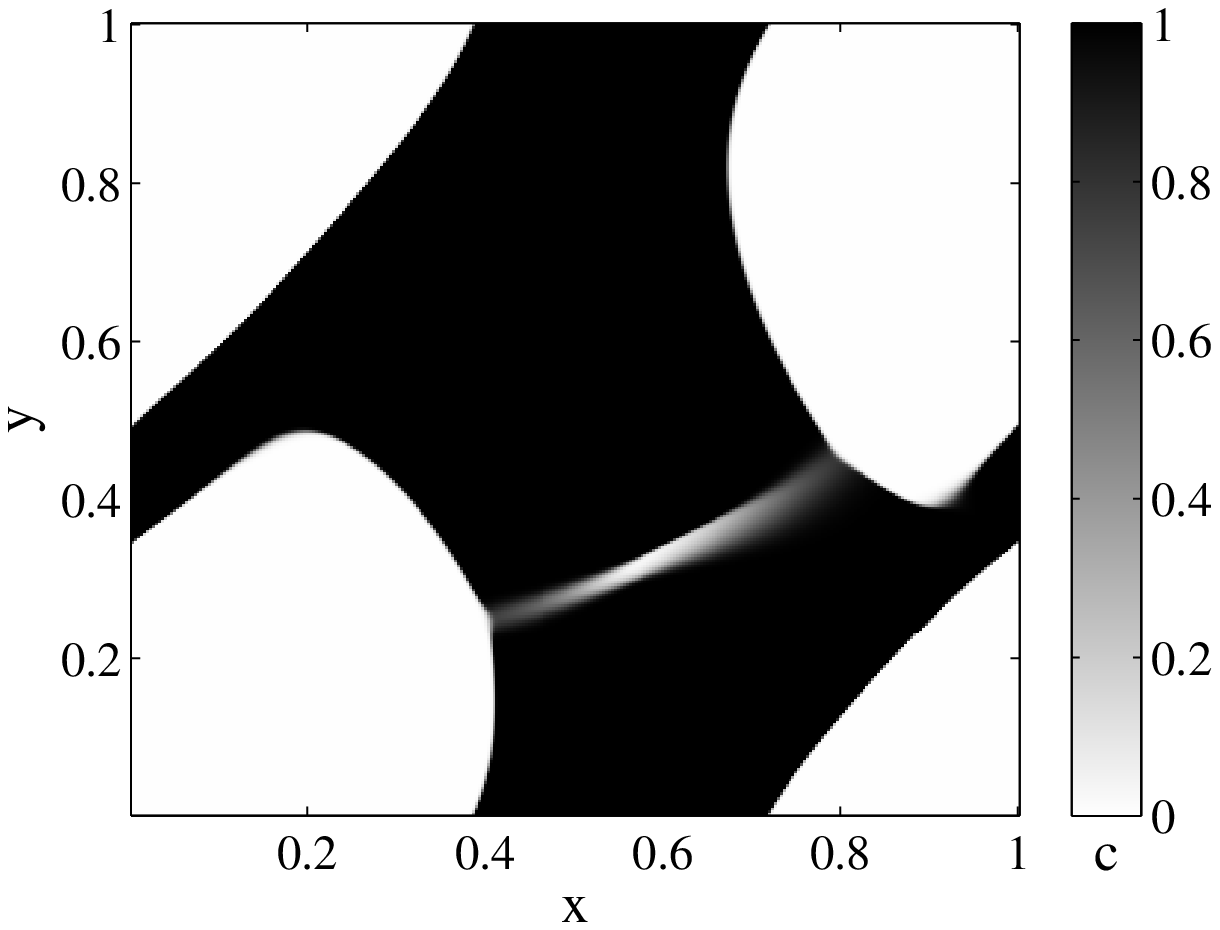} &  \includegraphics[width=6.4cm]{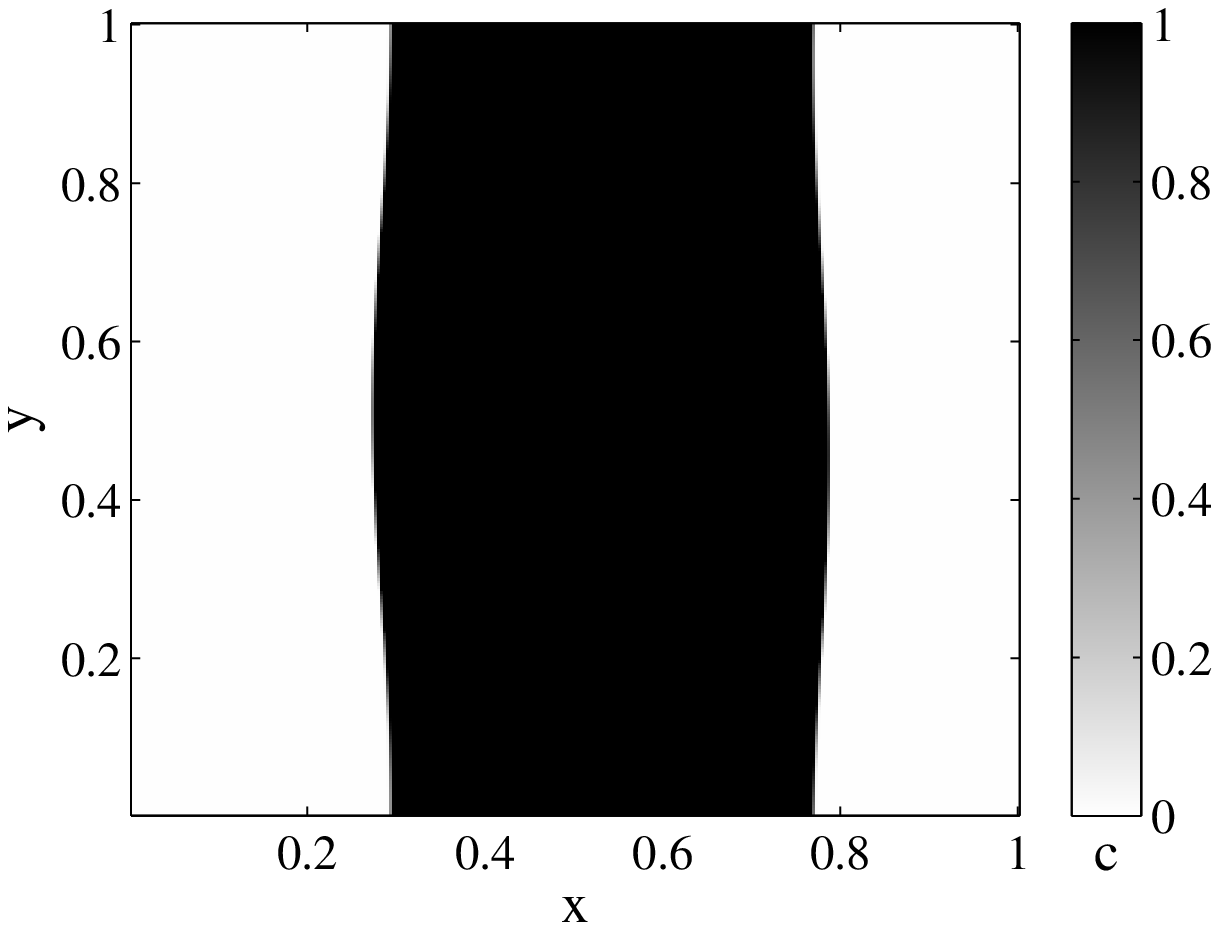}
\end{tabular}
\caption{Evolution of the concentration $\rs$ in the case of a desired velocity given by~\eqref{KS-1}-\eqref{KS-2}, with periodic  boundary conditions~\eqref{KS-2a}-\eqref{KS-2c}. The domain is initially filled randomly with 50\% moving species. Mesh resolution: $N_x \times N_y = 300 \times 300$.}\label{fig:KS-2}
\end{center}
\end{figure}

\section{Conclusion, extensions}
\label{sec:conc}

We proposed a model to describe the motion of mixtures of cell populations in a saturated medium with a constraint on the local  density. We provided an adapted theoretical framework, based on a reformulation of the model as a gradient flow in a product space of densities, and proposed a discretization strategy which enjoys reasonable stability and accuracy properties.

In terms of modelling, as mentioned in the introduction, any number $N$ of species can be handled, as far as equations are considered (possible issues concerning boundary conditions are disregarded here).  Saturation writes
simply $\rho_1+\dots+\rho_N = 1$,  we have an  advection equation for each species
$$
\partial _t \ri +\Div (\ri\left (\Vi+\w)\right)  =0\virg i=1,\dots, N,\\
$$
and the common correction velocity verifies
$$
\Div \w = -\Div \left (\sum_{i=1}^N \ri \Vi \right) .
$$
It could be of particular importance to include also proliferation phenomena. Denoting by $\beta_i$ the local rate of creation (possibly depending explicity upon $\rho_i$ or other densities), conservation equations become
$$
\partial _t \ri +\Div (\ri\left (\Vi+\w)\right)  =\beta_i\virg i=1,\dots, N,\\
$$
and the constraint on $\w$ accounts for mass creation:
$$
\Div \w = \sum_{i=1}^N \beta_i-\Div \left (\sum_{i=1}^N \ri \Vi \right) .
$$
Note that, in this situation, global non conservation rules out the use of no-flux conditions (or periodic setting). Considering a bounded domain $\Omega$ one can consider its boundary as a free outlet (interaction pressure  is set at $0$), so that a non conservative flux through boundary can compensate  the unbalance of mass in the domain. 
In case of  mass creation ($\beta_i >0$), one can  expect outflow through the boundary, so that no additional condition is needed for the advection  equations. In case some individual velocities $\V_i + \w$  may point inward the domain, the system has to be complemented with appropriate conditions  (e.g. prescribed value of ingoing densities).

As for theoretical issues, non conservation  rules out the  standard framework of gradient flow in the Wasserstein space, which is dedicated to  measures with constant total mass. Yet, as suggested in~\cite{MauRouSan}, generalizations of the JKO scheme, in the spirit of prediction-correction algorithms (or catching-up algorithms for sweeping processes, see for example~\cite{moreau}) might be expected to provide well-posedness results.

One may also wonder whether it could be possible to recover evolution equations for a single species, as in the case of crowd motion models (\cite{MauRouSan}). In this situation, species $2$ is replaced by empty space, which can be moved at no cost. Our proposed model could be seen as an attempt  to replace a unilateral constraint $\rho_1\leq 1$ (which is quite delicate to handle, see again~\cite{MauRouSan}), by an equality $\rho_1+\rho_2 = 1$ with $\rho_2\geq 0$.
It can be done formally by lowering the importance of species $2$, more precisely by  defining the correction velocity as the one which minimizes a weighted $L^2$ norm:
$$
\w = \hbox{argmin} \int_\Omega k_\varepsilon(\rho_1,\rho_2)\abs { \vv} ^2,
$$
where the minimum is taken among all those fields which satisfy $\Div \w = -\Div(\rho_1\Va)$ (in the case $\Vb = 0$), 
and 
$$
k_\varepsilon(\rho_1,\rho_2) = \rho_1 + \varepsilon \rho _2.
$$
One recovers formally an evolution equation for $\rho_1$ with unilateral constraint $\rho_1 \le 1$, yet with a significant difference: if one considers a saturated zone of $2$ surrounded by a saturated zone of $1$, the asymptotic ($\varepsilon\rightarrow 0$)  model  sees the inclusion as globally incompressible, which is not the case if one imposes simply $\rho_1 \leq 1$.

\medskip
Received xxxx 20xx; revised xxxx 20xx.
\medskip

\end{document}